\documentclass[a4paper,11pt]{amsart}

\usepackage{amsmath, amsthm, amsfonts, amssymb}
\usepackage[bookmarks=false]{hyperref}
\usepackage{enumerate}

\usepackage{color}

\numberwithin{equation}{section}

\newtheorem{letterthm}{Theorem}

\newtheorem{lettercor}[letterthm]{Corollary}

\newtheorem{thm}{Theorem}[section]
\newtheorem{lem}[thm]{Lemma}

\newtheorem{cor}[thm]{Corollary}
\newtheorem{prop}[thm]{Proposition}

\theoremstyle{definition}

\newtheorem{df}[thm]{Definition}

\newtheorem{claim}[thm]{Claim}

\newcommand{\R}{\mathbf{R}}
\newcommand{\C}{\mathbf{C}}
\newcommand{\Z}{\mathbf{Z}}
\newcommand{\F}{\mathbf{F}}

\newcommand{\N}{\mathbf{N}}

\newcommand{\Ad}{\operatorname{Ad}}
\newcommand{\id}{\text{\rm id}}

\newcommand{\Aut}{\mathord{\text{\rm Aut}}}
\newcommand{\rL}{\mathord{\text{\rm L}}}

\newcommand{\rC}{\mathord{\text{\rm C}}}

\newcommand{\rd}{\mathord{\text{\rm d}}}
\newcommand{\rT}{\mathord{\text{\rm T}}}

\newcommand{\rD}{\mathord{\text{\rm D}}}
\newcommand{\rE}{\mathord{\text{\rm E}}}

\newcommand{\Out}{\mathord{\text{\rm Out}}}

\newcommand{\core}{\mathord{\text{\rm c}}}
\newcommand{\Ind}{\mathord{\text{\rm Ind}}}

\newcommand{\Tr}{\mathord{\text{\rm Tr}}}
\newcommand{\Ball}{\mathord{\text{\rm Ball}}}
\newcommand{\spn}{\mathord{\text{\rm span}}}

\newcommand{\ovt}{\mathbin{\overline{\otimes}}}

\newcommand{\dpr}{^{\prime\prime}}

\begin{document}

\title[Factoriality, Connes' invariants and fullness of amalgamated free products]{Factoriality, Connes' type III invariants and fullness of amalgamated free product von Neumann algebras}

\begin{abstract}
We investigate factoriality, Connes'\ type ${\rm III}$ invariants and fullness of arbitrary amalgamated free product von Neumann algebras using Popa's deformation/rigidity theory. Among other things, we generalize many previous structural results on amalgamated free product von Neumann algebras and we obtain new examples of full amalgamated free product factors for which we can explicitely compute Connes' type ${\rm III}$ invariants.
\end{abstract}

\author{Cyril Houdayer}
\address{Laboratoire de Math\'ematiques d'Orsay\\ Universit\'e Paris-Sud\\ CNRS\\ Universit\'e Paris-Saclay\\ 91405 Orsay\\ FRANCE}
\email{cyril.houdayer@math.u-psud.fr}
\thanks{CH is supported by ERC Starting Grant GAN 637601}

\author{Yusuke Isono}
\address{RIMS, Kyoto University, 606-8502 Kyoto, JAPAN}
\email{isono@kurims.kyoto-u.ac.jp}
\thanks{YI is supported by JSPS Research Fellowship}


\maketitle

\section{Introduction and statement of the main results}

The amalgamated free product construction in operator algebras was introduced by Voiculescu in \cite{Vo85, VDN92}. It is one of the central objects of Voiculescu's free probability theory.  We briefly review the definition of the amalgamated free product in the framework of von Neumann algebras. Throughout this introduction, we use the following notation and terminology. For every $i \in \{1, 2\}$, let $B \subset M_i$ be any inclusion of $\sigma$-finite von Neumann algebras with faithful normal conditional expectation $\rE_i : M_i \to B$. The {\em amalgamated free product} $(M, \rE) = (M_1, \rE_1) \ast_B (M_2, \rE_2)$ is a pair of von Neumann algebra $M$ generated by $M_1$ and $M_2$ and faithful normal conditional expectation $\rE : M \to B$ such that $M_1$ and $M_2$ are {\em freely independent} with respect to $\rE$:
$$\rE(x_1 \cdots x_n) = 0 \; \text{ whenever } \; n \geq 1, \; x_j \in M_{i_j}^\circ \; \text{ and } \; i_1 \neq \cdots \neq  i_{n}$$
where $M_i^\circ = \ker(\rE_i)$. We refer to \cite{Vo85, Po90, Ue98} and Section \ref{section:results} for further details regarding the definition of $M$. When $B = \C 1$, then $\rE_{i} = \varphi_{i} 1$ with $\varphi_{i} \in (M_{i})_{\ast}$, $\rE = \varphi 1$ with $\varphi \in M_{\ast}$ and we say that $(M, \varphi) = (M_1, \varphi_1) \ast (M_2, \varphi_2)$ is the {\em free product} von Neumann algebra. In what follows, we simply write AFP for amalgamated free product.

The aim of this paper is to investigate factoriality, Connes'\ type ${\rm III}$ invariants and fullness of arbitrary AFP von Neumann algebras using Popa's deformation/rigidity theory. Among other things, we generalize many previous structural results on AFP von Neumann algebras \cite{Ue98, Ue00, Ue10, Ue11, Ue12} and we obtain new examples of full AFP factors for which we can explicitely compute Connes' type ${\rm III}$ invariants. Let us point out that the questions of factoriality, Connes' type ${\rm III}$ invariants and fullness for {\em plain} free product von Neumann algebras were completely settled by Ueda in \cite{Ue10, Ue11}.

\subsection*{Factoriality of AFP von Neumann algebras}

We introduce some further terminology. Following \cite[D\'efinition 1.3.1]{Co72} (see also \cite[Th\'eor\`eme 1.2.1]{Co72}), for any $\sigma$-finite von Neumann algebra $N$ and any faithful state $\varphi \in N_\ast$, Connes' invariant $\rT(N)$ is the subgroup of $\R$ defined as 
$$\rT(N) = \left \{ t \in \R \mid \exists u \in \mathcal U(N) \text{ such that } \sigma_t^\varphi = \Ad(u) \right \}.$$
We denote by $\core(N) = N \rtimes_{\sigma^{\varphi}} \R$ the {\em continuous core} of $N$ that, up to isomorphism, does not depend on the choice of the faithful state $\varphi \in N_{\ast}$. Following \cite[Definition 4.1]{Ue12}, we say that an inclusion of von Neumann algebras $Q \subset N$ is {\em entirely nontrivial} if no nonzero direct summand of $Q$ is a direct summand of $N$. Observe that when $Q = \C 1$, the inclusion $\C1 \subset N$ is entirely nontrivial if and only if $N \neq \C1$. We say that an inclusion of von Neumann algebras $Q \subset N$ is {\em with expectation} if there exists a faithful normal conditional expectation $\rE_{Q} : N \to Q$. Let now $N$ be any $\sigma$-finite von Neumann algebra and $P, Q \subset N$ any von Neumann subalgebras with expectation. Following \cite{Po01, Po03, HI15a}, we say that $P$ {\em embeds with expectation into} $Q$ {\em inside} $N$ and we write $P \preceq_{N} Q$ if a corner of $P$ can be conjugated with expectation into a corner of $Q$ via a nonzero partial isometry in $N$. We refer to Definition \ref{df-intertwining} for a more precise statement. Following \cite{OP07}, we say that $P$ is {\em amenable relative to} $Q$ {\em inside} $N$ if there exists a conditional expectation $\Phi : \langle N, Q \rangle \to P$ such that $\Phi |_{N}$ is faithful and normal.

Our first main result is a structure theorem for arbitrary AFP von Neumann algebras in which we investigate factoriality, computation of Connes' $\rT$ invariant, semifiniteness and amenability. 

\begin{letterthm}\label{thmA}
Keep the same notation as above. Assume that $M_1 \npreceq_{M_1} B$ and the inclusion $B \subset M_2$ is entirely nontrivial. Then the following assertions hold:

\begin{itemize}
\item [$(\rm i)$] $\mathcal Z(M) = M' \cap B$ and so $\mathcal Z(M) = \mathcal Z(M_1) \cap \mathcal Z(M_2) \cap \mathcal Z(B)$.

\item [$(\rm ii)$] For every faithful state $\varphi \in M_\ast$ such that $\varphi = \varphi \circ \rE$, we have 
$$\rT(M) = \left\{ t \in \R \mid \exists u \in \mathcal U(B) \text{ such that } \sigma_t^\varphi = \Ad(u) \right \}.$$

\item [$(\rm iii)$] $M$ is semifinite if and only if there exists a faithful normal semifinite trace $\Tr_{B}$ on $B$ such that $\Tr_{B} \circ \rE_i$ is a trace on $M_i$ for every $i \in \{1, 2\}$.

\item [$(\rm iv)$] For every nonzero projection $z \in \mathcal Z(M)$, $Mz$ is not amenable relative to $M_1$ inside $M$ and thus $Mz$ is not amenable.
\end{itemize} 

Assume moreover that $\core(M_{1}) \npreceq_{\core(M_{1})} \core(B)$. Then we have
\begin{itemize}

\item [$(\rm v)$]  $\mathcal Z(\core(M)) = \core(M)' \cap \core(B)$ and so $\mathcal Z(\core(M)) = \mathcal Z(\core(M_1)) \cap \mathcal Z(\core(M_2)) \cap \mathcal Z(\core(B))$.

\end{itemize} 
\end{letterthm}

Let us point out that the extra assumption that $\core(M_{1}) \npreceq_{\core(M_{1})} \core(B)$ in Theorem \ref{thmA} is satisfied whenever there exists a semifinite von Neumann subalgebra with expectation $A \subset M_{1}$ such that $A \npreceq_{M_{1}} B$ or if $B$ is of type ${\rm I}$ (see Proposition \ref{prop-intertwining-semifinite}). In that respect, Theorem \ref{thmA} above generalizes \cite[Theorem 4.3]{Ue12} where it is assumed that $M_{1}$ contains a finite von Neumann subalgebra with expectation $P \subset M_{1}$ such that $P \npreceq_{M_{1}} B$. Quite surprisingly, we also show in Theorem \ref{rem-intertwining} that there are natural examples of inclusions $Q \subset N$ where $N \npreceq_{N} Q$ while $\core(N) \preceq_{\core(N)} \core(Q)$. These examples demonstrate that when $P \subset N$ is a type ${\rm III}$ von Neumann subalgebra with expectation, there is no `good' analytic criterion in terms of nets of unitaries witnessing the fact that $P \npreceq_{N} Q$ as in \cite[Theorem 4.3 (5)]{HI15a}.

The proof of Theorem \ref{thmA} consists in generalizing a key result by Ioana--Peterson--Popa \cite[Theorem 1.1]{IPP05} about controlling relative commutants and normalizers in tracial AFP von Neumann algebras to {\em arbitrary} AFP von Neumann algebras (see Theorem \ref{thm-control}). However, as there is no good analytic criterion in terms of nets of unitaries witnessing the fact that $M_{1} \npreceq_{M_{1}} B$, we cannot generalize the proof of \cite[Theorem 1.1]{IPP05}. Instead, we use the notion of relative amenability via the existence of non normal conditional expectations as in \cite{BH16, Is17} and we apply the main result of \cite{BH16} regarding amenable absorption in arbitrary AFP von Neumann algebras.

\subsection*{Fullness of AFP von Neumann algebras}

Let $N$ be any $\sigma$-finite von Neumann algebra. Following \cite{Co74}, we say that a uniformly bounded net $(x_{j})_{j \in J}$ in $N$ is {\em centralizing} if $\lim_{j} \|x_{j} \zeta - \zeta x_{j}\| = 0$ for every $\zeta \in \rL^{2}(N)$ and {\em trivial} if there exists a bounded net $(\lambda_{j})_{j \in J}$ in $\C$ such that $x_{j} - \lambda_{j} 1 \to 0$ $\ast$-strongly. We say that $N$ is a {\em full} factor if every centralizing uniformly bounded net is trivial. This is equivalent to saying that for every nonempty directed set $J$ and every cofinal ultrafilter $\omega$ on $J$, the corresponding {\em asymptotic centralizer} $N_{\omega}$ is trivial, that is, $N_{\omega} = \C 1$. In that case, $\Out(N)$ is a complete topological group and Connes' invariant $\tau(N)$ is defined as the weakest topology on $\R$ that makes the modular homomorphism $\R \to \Out(N)$ continuous. When $N$ is a type ${\rm II_{1}}$ factor, $N$ is full if and only if $N$ does not have property Gamma of Murray--von Neumann \cite{MvN43} (see \cite[Corollary 3.8]{Co74}). 

The notion of fullness has always played a central role in von Neumann algebras. Murray--von Neumann \cite{MvN43} showed that free group factors are full, thus demonstrating that free group factors are not isomorphic to the hyperfinite type ${\rm II_{1}}$ factor. Connes \cite{Co75b} proved a spectral gap characterization of full type ${\rm II_{1}}$ factors that was a crucial ingredient in his proof of the uniqueness of the injective type ${\rm II_{1}}$ factor (we refer to \cite{Ma16, HMV16} for recent spectral gap characterizations of full type ${\rm III}$ factors). Over the last 15 years, the notion of fullness has been a key tool in proving various rigidity results in the framework of Popa's deformation/rigidity theory (see e.g.\ \cite{Pe06, Po06, Io12, HI15b, HU15b, HMV16}). For all these reasons, it is natural to investigate the fullness property for arbitrary AFP von Neumann algebras.

In order to understand the position of centralizing nets in $M = M_{1} \ast_{B} M_{2}$, Ueda  \cite{Ue98, Ue00, Ue10, Ue11, Ue12} generalized Popa's asymptotic orthogonality property \cite{Po83} for free group factors to AFP von Neumann algebras. For this method to work though, one needs to assume not only that $M_{1} \npreceq_{M_{1}} B$ but also that $M_{1}$ has `many' unitaries that are perpendicular to $B$ with respect to $\rE_{1} : M_{1} \to B$ and that these unitaries lie in the centralizer of a fixed faithful normal state on $M_{1}$. In order to understand the position of centralizing nets in $M = M_{1} \ast_{B} M_{2}$, we use a completely different method based on Popa's deformation/rigidity theory. As we will see, our method allows us to generalize many of the results obtained by Ueda in \cite{Ue98, Ue00, Ue10, Ue11, Ue12}. We proceed in two steps.

Firstly, we prove a {\em spectral gap} rigidity result for arbitrary von Neumann subalgebras $P \subset M$ with expectation in the spirit of \cite[Theorem 6.3]{Io12} (see also \cite[Theorem 1.5]{Po06}, \cite[Theorem 4.3]{Pe06} and \cite[Theorem A]{HI15b} for related results). 

\begin{letterthm}\label{thmB}
Keep the same notation as above. Let $P \subset M$ be any von Neumann subalgebra with expectation. Then at least one of the following assertions holds true:

\begin{itemize}
\item [$(\rm i)$] For every nonempty directed set $J$ and every cofinal ultrafilter ${\omega}$ on $J$, we have
$$P_{\omega} \subset M \vee B^{\omega} \quad \quad \text{and} \quad \quad \core(P)_\omega \subset \core(M) \vee \core(B)^\omega.$$

\item [$(\rm ii)$] There exists $i \in \{1, 2\}$ such that $\core(P) \preceq_{\core(M)} \core(M_{i})$.

\item [$(\rm iii)$] There exists a nonzero projection $z \in \mathcal Z(\core(P)' \cap \core(M))$ such that $\core(P)z$ is amenable relative to $\core(B)$ inside $\core(M)$.
\end{itemize}
\end{letterthm}

Note that the inclusion of continuous cores $\core(P) \subset \core(M)$ is the one associated with a fixed faithful normal conditional expectation $\rE_{P} : M \to P$. The proof of Theorem \ref{thmB} generalizes the one of \cite[Theorem 6.3]{Io12} and relies on a combination of Popa's deformation/rigidity theory and Connes--Tomita--Takesaki modular theory. 

Secondly, to understand the position of centralizing nets in $M = M_{1} \ast_{B} M_{2}$, we apply Theorem \ref{thmB} to $P = M$. We want to use Assertion $(\rm i)$ and exploit the key fact that $M \vee B^{\omega}$ (resp.\ $\core(M) \vee \core(B)^\omega$) can be regarded as an AFP von Neumann algebra. Assuming that $M_{1} \npreceq_{M_{1}} B$ and the inclusion $B \subset M_{2}$ is entirely nontrivial as in Theorem \ref{thmA}, we prove that Assertion $(\rm iii)$ in Theorem \ref{thmB} cannot occur. On the other hand, we improve Assertion $(\rm ii)$ in Theorem \ref{thmB} by showing that $\core(M_{1}) \preceq_{\core(M_{1})} \core(B)$. However, as we already pointed out, this does not contradict our assumption that $M_{1} \npreceq_{M_{1}} B$. Thus, in order to obtain the fullness property for AFP von Neumann algebras, we need to add extra assumptions.

When the amalgam $B$ is a type ${\rm I}$ von Neumann algebra, we obtain the following sharp rigidity result for the asymptotic centralizer $M_\omega$, generalizing \cite[Theorems 4.8 and 4.10]{Ue12}.

\begin{letterthm}\label{thmC}
Keep the same notation as above. Assume that $B$ is a type ${\rm I}$ von Neumann algebra, $M_1 \npreceq_{M_1} B$ and the inclusion $B \subset M_2$ is entirely nontrivial.

Then for every nonempty directed set $J$ and every cofinal ultrafilter $\omega$ on $J$, we have  
$$M_\omega = M' \cap \mathcal Z(B)^\omega \quad \quad \text{and} \quad \quad \core(M)_\omega = \core(M)' \cap \mathcal Z(\core(B))^\omega.$$

Furthermore, if $M$ has separable predual and $M' \cap \mathcal Z(B)^\omega = \C 1$ for some $\omega \in \beta(\N) \setminus \N$, that is, $M$ is full, then for every faithful state $\varphi \in M_{\ast}$ such that $\varphi \circ \rE = \varphi$ and every sequence $(t_{n})_{n}$ in $\R$, the following conditions are equivalent:
\begin{itemize}
\item $t_{n} \to 0$ with respect to $\tau(M)$.
\item There exists a sequence $v_{n} \in \mathcal U(B)$ such that $\Ad(v_{n}) \circ \sigma_{t_{n}}^{\varphi} \to \id_{M}$ in $\Aut(M)$.
\end{itemize}
\end{letterthm}

In particular, any such $M$ (resp.\ $\core(M)$) as in Theorem \ref{thmC} does not tensorially absorb the hyperfinite type ${\rm II_1}$ factor, that is, $M$ (resp.\ $\core(M)$) is not McDuff \cite{McD69, Co75a}.

We say that a nonsingular equivalence relation $\mathcal R$ (with countable classes)  defined on a standard measure space $(X, \mu)$ is {\em recurrent} if for every measurable subset $W \subset X$ with $\mu(W) > 0$ and for $\mu$-almost every $x \in W$, we have $|W \cap [x]_{\mathcal R}| = + \infty$. Observe that the nonsingular equivalence relation $\mathcal R$ is recurrent if and only if its von Neumann algebra $\rL(\mathcal R)$ has no nonzero type ${\rm I}$ direct summand (see \cite{FM75} for the construction of the von Neumann algebra $\rL(\mathcal R)$). Likewise, we say that a nonsingular action $\Gamma \curvearrowright (X, \mu)$ of a countable discrete group $\Gamma$ on a standard measure space $(X, \mu)$ is {\em recurrent} if for every measurable subset $W \subset X$ with $\mu(W) > 0$ and for $\mu$-almost every $x \in W$, we have $|\{g \in \Gamma \mid g \cdot x \in W\}| = +\infty$. For any nonsingular equivalence relation $\mathcal R$ (with countable classes) defined on a standard measure space $(X, \mu)$, we denote by $\core(\mathcal R)$ its {\em Maharam extension}. We then canonically have $\rL(\core(\mathcal R)) = \core(\rL(\mathcal R))$ (see e.g.\ \cite{HMV17}).

As a straightforward consequence of Theorem \ref{thmC}, we obtain the fullness property for factors $\rL(\mathcal R)$ arising from strongly ergodic AFP equivalence relations $\mathcal R$. In that case, we also infer that the $\tau$ invariant of the strongly ergodic equivalence relation $\mathcal R$ as defined in \cite{HMV17} coincides with Connes' $\tau$ invariant of the full factor $\rL(\mathcal R)$ \cite{Co74}. We refer to \cite[D\'efinition IV.6]{Ga99} for the notion of AFP in the framework of nonsingular equivalence relations.

\begin{lettercor}\label{corD} 
Let $\mathcal R$ be any nonsingular equivalence relation with countable classes defined on a diffuse standard measure space $(X, \mu)$ that splits as an AFP equivalence relation $\mathcal R = \mathcal R_{1} \ast_{\mathcal S} \mathcal R_{2}$. Assume that the following properties hold:
\begin{itemize}
\item [(P$1$)] $\mathcal S$ is of type ${\rm I}$.
\item [(P$2$)] $\mathcal R_{1}$ is recurrent. 
\item [(P$3$)] For $\mu$-almost every $x \in X$, we have $|[x]_{{\mathcal R_{2}}}|/|[x]_{\mathcal S}| \geq 2$.
\end{itemize}
Then for every nonempty directed set $J$ and every cofinal ultrafilter $\omega$ on $J$, we have  
$$\rL(\mathcal R)_\omega = \rL(\mathcal R)' \cap \mathcal Z(\rL(\mathcal S))^\omega \quad \quad \text{and} \quad \quad \rL(\core(\mathcal R))_\omega = \rL(\core(\mathcal R))' \cap \mathcal Z(\rL(\core(\mathcal S)))^\omega.$$
In particular, we have that $\mathcal R$ is strongly ergodic if and only if $\rL(\mathcal R)$ is a full factor.
If $\mathcal R$ is strongly ergodic, we have $\tau(\mathcal R) = \tau(\rL(\mathcal R))$.
\end{lettercor}

In particular, any such $\mathcal R$ (resp.\ $\core(\mathcal R)$) as in Corollary \ref{corD} does not absorb (by direct product) the hyperfinite type ${\rm II_1}$ equivalence relation, that is, $\mathcal R$ (resp.\ $\core(\mathcal R)$) is not stable in the sense of \cite{JS85}. Let us point out that Corollary \ref{corD} generalizes \cite[Corollary 10]{Ue00}. We refer to \cite[Theorem C]{HI15b} and \cite[Theorem E]{HMV17} for related results regarding nonsingular equivalence relations arising from strongly ergodic actions of hyperbolic groups.

When the amalgam $B$ is an arbitrary von Neumann algebra with separable predual and $M = B \rtimes \Gamma$ is the crossed product von Neumann algebra arising from an arbitrary properly outer action $\Gamma \curvearrowright B$ of a free product group $\Gamma = \Gamma_{1} \ast \Gamma_{2}$ where $|\Gamma_{1}| = +\infty$ and $|\Gamma_{2}| \geq 2$, we obtain the following new rigidity result for the asymptotic centralizer $M_{\omega}$.

\begin{letterthm}\label{thmE}
Let $\Gamma_1$ and $\Gamma_{2}$ be any countable discrete groups such that $|\Gamma_{1}| = +\infty$ and $|\Gamma_{2}| \geq 2$. Put $\Gamma = \Gamma_{1} \ast \Gamma_{2}$. Let $B$ be any von Neumann algebra with separable predual and $\Gamma \curvearrowright B$ any properly outer action such that $\Gamma_{1} \curvearrowright \mathcal Z(B)$ is recurrent. Put $M = B \rtimes \Gamma$.

\begin{itemize}
\item [$(\rm i)$] Assume that $\Gamma_{2}$ is finite or $\Gamma_{2} \curvearrowright \mathcal Z(B)$ is recurrent. 
Then for every nonempty directed set $J$ and every cofinal ultrafilter $\omega$ on $J$, we have  
$$M_\omega \subset M' \cap B^\omega.$$

\item [$(\rm ii)$] Assume that $\Gamma_{1} \curvearrowright \core(B)$ is properly outer and $\Gamma_{1} \curvearrowright \mathcal Z(\core(B))$ is recurrent.
Then for every nonempty directed set $J$ and every cofinal ultrafilter $\omega$ on $J$, we have  
$$M_\omega \subset M' \cap B^\omega  \quad \quad \text{and} \quad \quad \core(M)_\omega \subset \core(M)' \cap \core(B)^\omega.$$
Furthermore, if $M' \cap B^\omega = \C 1$ for some $\omega \in \beta(\N) \setminus \N$, that is, $M$ is full, then for every faithful state $\varphi \in M_{\ast}$ such that $\varphi \circ \rE = \varphi$ and every sequence $(t_{n})_{n}$ in $\R$, the following conditions are equivalent:

\begin{itemize}

\item [$\bullet$] $t_{n} \to 0$ with respect to $\tau(M)$.

\item [$\bullet$] There exists a sequence $v_{n} \in \mathcal U(B)$ such that $\Ad(v_{n}) \circ \sigma_{t_{n}}^{\varphi} \to \id_{M}$ in $\Aut(M)$.

\end{itemize}

\end{itemize}
\end{letterthm}

We can apply Theorem \ref{thmE} to obtain new examples of full factors of type ${\rm III}$ for which we can compute Connes' $\tau$ invariant.

\begin{lettercor}\label{corF}
Let $\Gamma_1$ and $\Gamma_{2}$ be any countable discrete groups such that $|\Gamma_{1}| = +\infty$ and $|\Gamma_{2}| \geq 2$. Put $\Gamma = \Gamma_{1} \ast \Gamma_{2}$. Let $B$ be any full factor with separable predual and $\Gamma \curvearrowright B$ any outer action. Then $M = B \rtimes \Gamma$ is a full factor.

If moreover $B$ is of type ${\rm III_{1}}$ and $\Gamma_{1} \curvearrowright \core(B)$ is outer, then $M$ is of type ${\rm III_{1}}$ and $\tau(M) = \tau(B)$.
\end{lettercor}

Let us point out that Corollary \ref{corF} applies in particular to free groups $\F_{n}$ where $n \geq 2$ or $n = +\infty$. Corollary \ref{corF} should be compared to \cite[Theorem B]{Ma16} where it is shown that for any full factor $B$, any discrete group $\Gamma$ and any outer action $\Gamma \curvearrowright B$ such that the image of $\Gamma$ is discrete in $\Out(B)$, the crossed product factor $B \rtimes \Gamma$ is full. 

\subsection*{Acknowledgments}
Y.\ Isono is grateful to K.\ Hasegawa for insightful discussions regarding the proof of Theorem \ref{thm-crossed-products}.

\tableofcontents

\section{Preliminaries}

\subsection*{Background on $\sigma$-finite von Neumann algebras}

For any von Neumann algebra $M$, we denote by $\mathcal Z(M)$ its centre, by $\mathcal U(M)$ its group of unitaries, by $\Ball(M)$ its unit ball with respect to the uniform norm $\|\cdot\|_\infty$ and by $(M, \rL^2(M), J^{M}, \rL^2(M)^+)$ its standard form \cite{Ha73}. When no confusion is possible, we simply write $J = J^{M}$. 

Let $M$ be any $\sigma$-finite von Neumann algebra and $\varphi \in M_\ast$ any faithful state. We  write $\|x\|_\varphi = \varphi(x^* x)^{1/2}$ for every $x \in M$. Recall that on $\Ball(M)$, the topology given by $\|\cdot\|_\varphi$ 
coincides with the $\sigma$-strong topology. We denote by $\xi_\varphi = \varphi^{1/2} \in \rL^2(M)^+$ the unique element such that $\varphi = \langle \, \cdot \, \xi_{\varphi}, \xi_{\varphi}\rangle$. Conversely, for every $\xi \in \rL^{2}(M)^{+}$, we denote by $\varphi_{\xi} = \langle \, \cdot \, \xi, \xi\rangle \in M_{\ast}$ the corresponding positive form. The mapping $M \to \rL^2(M) : x \mapsto x \xi_\varphi$ defines an embedding with dense image such that $\|x\|_\varphi = \|x \xi_\varphi\|$ for all $x \in M$. 

\subsection*{Jones basic construction}

Let $B \subset M$ be any inclusion of $\sigma$-finite von Neumann algebras with faithful normal conditional expectation $\rE_B : M \to B$. We regard the standard form of $B$ as a substandard form of the standard form of $M$ via the isometric embedding $\rL^{2}(B)^{+} \to \rL^{2}(M)^{+} : \xi \mapsto (\varphi_{\xi} \circ \rE_{B})^{1/2}$.  
Fix a faithful state $\varphi_B \in B_{\ast}$ and put $\varphi=\varphi_B \circ \rE_B \in M_{\ast}$. The Jones projection $e_{B} : \rL^{2}(M) \to \rL^{2}(B)$ satisfies the relation $e_{B}(x \xi_{\varphi}) = \rE_{B}(x)\xi_{\varphi}$ for every $x \in M$. Moreover, $e_{B}$ does not depend on the choice of the faithful state $\varphi_B \in B_{\ast}$ (see e.g.\ \cite[Proposition A.2]{HI15a}). The Jones {\em basic construction} for the inclusion $B \subset M$ is defined by $\langle M, B\rangle = (JBJ)' \cap \mathbf B(\rL^{2}(M))$. Observe that $\langle M, B\rangle$ does not depend on the choice of the faithful normal conditional expectation $\rE_{B} : M \to B$. We moreover have $\langle M, B\rangle = (M \cup \{e_{B}\})\dpr$.

\subsection*{Continuous core decomposition}

Let $M$ be any $\sigma$-finite von Neumann algebra and $\varphi \in M_\ast$ any faithful state. We denote by $\sigma^\varphi$ the modular automorphism group of the state $\varphi$.  The centralizer $M^\varphi$ of the state $\varphi$ is by definition the fixed point algebra of $(M, \sigma^\varphi)$.  The {\em continuous core} of $M$ with respect to $\varphi$, denoted by $\core_\varphi(M)$, is the crossed product von Neumann algebra $M \rtimes_{\sigma^\varphi} \R$.  The natural inclusion $\pi_\varphi: M \to \core_\varphi(M)$ and the unitary representation $\lambda_\varphi: \R \to \core_\varphi(M)$ satisfy the {\em covariance} relation
$$
\forall x \in M, \forall t \in \R, \quad  \lambda_\varphi(t) \pi_\varphi(x) \lambda_\varphi(t)^*
  =
  \pi_\varphi(\sigma^\varphi_t(x)).
$$
Put $\rL_\varphi (\R) = \lambda_\varphi(\R)\dpr$. There is a unique faithful normal conditional expectation $\rE_{\rL_\varphi (\R)}: \core_{\varphi}(M) \to \rL_\varphi(\R)$ satisfying $\rE_{\rL_\varphi (\R)}(\pi_\varphi(x) \lambda_\varphi(t)) = \varphi(x) \lambda_\varphi(t)$. The faithful normal semifinite weight defined by $f \mapsto \int_{\R} \exp(-s)f(s) \, {\rm d}s$ on $\rL^\infty(\R)$ gives rise to a faithful normal semifinite weight $\Tr_\varphi$ on $\rL_\varphi(\R)$ via the Fourier transform. The formula $\Tr_\varphi = \Tr_\varphi \circ \rE_{\rL_\varphi (\R)}$ extends it to a faithful normal semifinite trace on $\core_\varphi(M)$.

Because of Connes' Radon--Nikodym cocycle theorem \cite[Th\'eor\`eme 1.2.1]{Co72} (see also \cite[Theorem VIII.3.3]{Ta03}), the semifinite von Neumann algebra $\core_\varphi(M)$ together with its trace $\Tr_\varphi$ does not depend on the choice of $\varphi$ in the following precise sense. If $\psi$ is another faithful normal state on $M$, there is a canonical surjective $\ast$-isomorphism
$\Pi_{\varphi,\psi} : \core_\psi(M) \to \core_{\varphi}(M)$ such that $\Pi_{\varphi,\psi} \circ \pi_\psi = \pi_\varphi$ and $\Tr_\varphi \circ \Pi_{\varphi,\psi} = \Tr_\psi$. Note however that $\Pi_{\varphi,\psi}$ does not map the subalgebra $\rL_\psi(\R) \subset \core_\psi(M)$ onto the subalgebra $\rL_\varphi(\R) \subset \core_\varphi(M)$ (and thus we use the symbol $\rL_\varphi(\R)$ instead of the usual $\rL(\R)$).

\subsection*{Ultraproduct von Neumann algebras}

Let $M$ be any $\sigma$-finite von Neumann algebra. Let $J$ be any nonempty directed set and ${\omega}$ any {\em cofinal} ultrafilter on $J$, that is, for all $j_0 \in J$, we have $\{j \in J : j \geq j_0\} \in \omega$. Define
\begin{align*}
\mathfrak I_{\omega}(M) &= \left\{ (x_j)_j \in \ell^\infty(J, M) : \lim_{j \to \omega} \|x_j \zeta\| = \lim_{j \to \omega} \|\zeta x_{j}\| =  0, \forall \zeta \in \rL^{2}(M) \right\} \\
\mathfrak M_{\omega}(M) &= \left\{ (x_j)_j \in \ell^\infty(J, M) : \lim_{j \to \omega}\|x_j \zeta - \zeta x_j \| = 0, \forall \zeta \in \rL^{2}(M) \right\} \\
\mathfrak M^{\omega}(M) &= \left\{ (x_j)_j \in \ell^\infty(J, M) :  (x_j)_j \, \mathfrak I_{\omega}(M) \subset \mathfrak I_{\omega}(M) \text{ and } \mathfrak I_{\omega}(M) \, (x_j)_j \subset \mathfrak I_{\omega}(M)\right\}.
\end{align*}
Observe that $\mathfrak I_{\omega}(M)  \subset \mathfrak M_{\omega}(M) \subset \mathfrak M^{\omega}(M)$. The {\em multiplier algebra} $\mathfrak M^{\omega}(M)$ is a $\rC^*$-algebra and $\mathfrak I_{\omega}(M) \subset \mathfrak M^{\omega}(M)$ is a norm closed two-sided ideal. Following \cite[\S 5.1]{Oc85}, we define the {\em ultraproduct von Neumann algebra} by $M^{\omega} = \mathfrak M^{\omega}(M) / \mathfrak I_{\omega}(M)$, which is indeed known to be a von Neumann algebra. Observe that the proof given in \cite[5.1]{Oc85} for the case when $J = \N$ and $\omega \in \beta(\N) \setminus \N$ applies {\em mutatis mutandis}. We denote the image of $(x_j)_j \in \mathfrak M^{\omega}(M)$ by $(x_j)^{\omega} \in M^{\omega}$. Following \cite[\S 2]{Co74}, we define the {\em asymptotic centralizer von Neumann algebra} by $M_{\omega} = \mathfrak M_{\omega}(M) / \mathfrak I_{\omega}(M)$, which is a von Neumann subalgebra of $M^{\omega}$. 

For every $x \in M$, the constant sequence $(x)_j$ lies in the multiplier algebra $\mathfrak M^{\omega}(M)$. We then identify $M$ with $(M + \mathfrak I_{\omega}(M))/ \mathfrak I_{\omega}(M)$ and regard $M \subset M^{\omega}$ as a von Neumann subalgebra. The map $\rE_{\omega} : M^{\omega} \to M : (x_j)^{\omega} \mapsto \sigma \text{-weak} \lim_{j \to {\omega}} x_j$ is a faithful normal conditional expectation. For every faithful state $\varphi \in M_\ast$, the formula $\varphi^{\omega} = \varphi \circ \rE_{\omega}$ defines a faithful normal state on $M^{\omega}$. Observe that $\varphi^{\omega}((x_j)^{\omega}) = \lim_{j \to {\omega}} \varphi(x_j)$ for all $(x_j)^{\omega} \in M^{\omega}$. By \cite[Proposition 2.8]{Co74} (see also \cite[Proposition 4.35]{AH12}), we have $M_\omega = (M' \cap M^\omega)^{\varphi^\omega}$ for every faithful state $\varphi \in M_\ast$.

Let $Q \subset M$ be any von Neumann subalgebra with faithful normal conditional expectation $\rE_Q : M \to Q$. Choose a faithful state $\varphi \in M_\ast$ such that $\varphi = \varphi \circ \rE_Q$. We have $\ell^\infty(J, Q) \subset \ell^\infty(J, M)$, $\mathfrak I_{\omega}(Q) \subset \mathfrak I_{\omega}(M)$ and $\mathfrak M^{\omega}(Q) \subset \mathfrak M^{\omega}(M)$. We then identify $Q^{\omega} = \mathfrak M^{\omega}(Q) / \mathfrak I_{\omega}(Q)$ with $(\mathfrak M^{\omega}(Q) + \mathfrak I_{\omega}(M)) / \mathfrak I_{\omega}(M)$ and regard $Q^{\omega} \subset M^{\omega}$ as a von Neumann subalgebra. Observe that the norm $\|\cdot\|_{(\varphi |_Q)^{\omega}}$ on $Q^{\omega}$ is the restriction of the norm $\|\cdot\|_{\varphi^{\omega}}$ to $Q^{\omega}$. Observe moreover that $(\rE_Q(x_j))_j \in \mathfrak I_{\omega}(Q)$ for all $(x_j)_j \in \mathfrak I_{\omega}(M)$ and $(\rE_Q(x_j))_j \in \mathfrak M^{\omega}(Q)$ for all $(x_j)_j \in \mathfrak M^{\omega}(M)$. Therefore, the mapping $\rE_{Q^{\omega}} : M^{\omega}\to Q^{\omega} : (x_j)^{\omega} \mapsto (\rE_Q(x_j))^{\omega}$ is a well-defined conditional expectation satisfying $\varphi^{\omega} \circ \rE_{Q^{\omega}} = \varphi^{\omega}$. Hence, $\rE_{Q^{\omega}} : M^{\omega} \to Q^{\omega}$ is a faithful normal conditional expectation.

\section{Relative amenability}

Let $M$ be any $\sigma$-finite von Neumann algebra and $A \subset 1_A M 1_A$ and $B \subset M$ any von Neumann subalgebras with expectation. Following \cite{OP07}, we say that $A$ is {\em amenable relative to $B$ inside $M$} and write $A \lessdot_M B$ if there exists a conditional expectation $\Phi : 1_A \langle M, B \rangle 1_A \to A$ such that the restriction $\Phi |_{1_A M 1_A} : 1_A M 1_A \to A$ is faithful and normal. Fix now $\rE_A : 1_A M 1_A \to A$ a faithful normal conditional expectation. We write $(A, \rE_A) \lessdot_M B$ if there exists a conditional expectation $\Phi : 1_A \langle M, B \rangle 1_A \to A$ such that $\Phi |_{1_A M 1_A} = \rE_A$.

We start by reviewing a useful characterization of relative amenability in arbitrary von Neumann algebras due to Isono \cite{Is17}.

\begin{thm}[{\cite[Theorem 3.2]{Is17}}]\label{thm-characterization}
Let $M$ be any $\sigma$-finite von Neumann algebra and $A \subset 1_{A}M 1_{A}$ and $B \subset M$ any von Neumann subalgebras with faithful normal conditional expectations $\rE_A : 1_{A}M1_{A} \to A$ and $\rE_B : M \to B$. Choose faithful states $\varphi_A, \varphi_B \in M_\ast$ such that $1_{A} \in M^{\varphi_{A}}$, $\varphi_A (1_{A} \, \cdot \, 1_{A})= \varphi_A \circ \rE_A$ and $\varphi_B = \varphi_B \circ \rE_B$. Simply write $\core(A) = \Pi_{\varphi_B, \varphi_A}(\core_{\varphi_A}(A))$, $\core(B) = \core_{\varphi_B}(B)$, $\core(M) = \core_{\varphi_B}(M)$. 

The following conditions are equivalent:
\begin{itemize}
\item [$(\rm i)$] $(A, \rE_A) \lessdot_M B$.
\item [$(\rm ii)$] $\core(A) \lessdot_{\core(M)} \core(B)$.
\item [$(\rm iii)$] There exists a ucp map $\Theta : 1_{A}\langle M, B\rangle1_{A} \to \langle 1_{A}M1_{A}, A\rangle$ such that $\Theta(x) = x$ for every $x \in 1_{A}M1_{A}$.
\end{itemize}
\end{thm}

As pointed out in \cite[Remark 3.3]{Is17}, condition $(\rm iii)$ in Theorem \ref{thm-characterization} above clearly does not depend on the choice of $\rE_{A}$. It follows that the notion of relative amenability $A \lessdot_{M} B$ does not depend on the choice of $\rE_{A}$. In particular, if $A \lessdot_M B$ then $(A, \rE_A) \lessdot_M B$ for every faithful normal conditional expectation $\rE_{A} : 1_{A }M 1_{A} \to A$. 

As a straightforward consequence of Theorem \ref{thm-characterization}, the notion of relative amenability in arbitrary von Neumann algebras is transitive (see \cite[Proposition 2.4 (3)]{OP07} for the tracial case).

\begin{cor}[{\cite[Corollary 3.4]{Is17}}]\label{cor-transitivity}
Let $M$ be any $\sigma$-finite von Neumann algebra and $A, B, C \subset M $ any von Neumann subalgebras with expectation. If $A \lessdot_{M} B$ and $B \lessdot_{M} C$, then $A \lessdot_{M} C$.
\end{cor}

We prove below a series of elementary yet useful results regarding stability properties of the notion of relative amenability in arbitrary von Neumann algebras.

\begin{lem}\label{lem-elementary-properties1}
Let $M$ be any $\sigma$-finite von Neumann algebra and $A \subset 1_A M 1_A$ and $B \subset M$ any von Neumann subalgebras with expectation. The following assertions hold:
\begin{itemize}
	\item[$\rm (i)$] Let $p,p'\in A$ (resp.\ $q,q'\in A'\cap 1_A M1_A$) be any projections that are equivalent in $A$ (resp.\ in $A'\cap 1_A M1_A$). We have that $pAp q \lessdot_M B$ if and only if $p'Ap' q' \lessdot_M B$.
	\item[$\rm (ii)$] Let $H$ be any separable Hilbert space. If $A \lessdot_M B$, then $A\ovt \mathbf B(H) \lessdot_{M\ovt \mathbf B(H)} B\ovt \mathbf B(H)$.
	\item[$\rm (iii)$] Let $p\in A$ and $q \in A'\cap 1_A M1_A$ be any projections such that $pq\neq 0$. If $A \lessdot_M B$, then $pApq \lessdot_M B$.
	\item[$\rm (iv)$] Let $p\in A$ and $q \in A'\cap 1_A M1_A$ be any projections such that $pq\neq 0$. Let $z_p\in \mathcal{Z}(A)$ and $z_q\in \mathcal{Z}(A'\cap 1_A M1_A)$ be the central support projections of $p$ and $q$ respectively.  If $pApq \lessdot_M B$, then $Az_pz_q \lessdot_M B$.
\item [$(\rm v)$] Let $(p_i)_{i \in I}$ be any family of projections in $A$ and denote by $p = \bigvee_{i \in I} p_i$. If $Ap_i \lessdot_M B$ for every $i \in I$, then $Ap \lessdot_M B$. The same conclusion holds for every family of projections $(p_i)_{i \in I}$ in $A'\cap 1_A M1_A$.
\end{itemize}
\end{lem}

\begin{proof}
Before starting the proof, we mention that for any projection $q\in A'\cap 1_AM1_A$, the inclusion $Aq \subset qMq$ is with expectation. This follows from \cite[Proposition 2.2]{HU15b}. In the proof of item $(\rm iii)$ below, we will actually construct a faithful normal conditional expectation for $Aq \subset qMq$ in a more direct way.

$(\rm i)$ This is trivial. 

$(\rm ii)$ It suffices to prove the statement when $H = \ell^{2}$. Assume that $A \lessdot_M B$ and fix a conditional expectation $\Phi : 1_A \langle M, B\rangle 1_A \to A$ such that $\Phi|_{1_AM1_A}$ is faithful and normal. Denote by $(\delta_k)_{k \in \N}$ the canonical orthonormal basis of $\ell^2$. For every $K \in \N$, denote by $q_K : \ell^2 \to \spn \left \{\delta_k : 0 \leq k \leq K \right\}$ the orthogonal projection. We identify $\mathbf M_{K + 1}(\C) = q_K \mathbf B(\ell^2)q_K$. Choose $\omega \in \beta(\N) \setminus \N$ and define the conditional expectation $\Phi^\omega : 1_A \langle M, B\rangle 1_A \ovt \mathbf B(\ell^2) \to A \ovt \mathbf B(\ell^2)$ by the formula $$\forall x \in 1_A \langle M, B\rangle 1_A \ovt \mathbf B(\ell^2), \quad \Phi^\omega(x) = \sigma\text{-weak} \lim_{K \to \omega} (\Phi \otimes \id_{\mathbf M_{K + 1}(\C)})((1 \otimes q_K) x (1 \otimes q_K)).$$  Moreover, we have $\Phi^\omega |_{1_AM1_A \ovt \mathbf B(\ell^2)} = \Phi|_{1_AM1_A} \ovt \id_{\mathbf B(\ell^2)}$ and thus $\Phi^\omega |_{1_AM1_A \ovt \mathbf B(\ell^2)}$ is faithful and normal. This shows that $A\ovt \mathbf B(\ell^{2}) \lessdot_{M\ovt \mathbf B(\ell^{2})} B\ovt \mathbf B(\ell^{2})$.

$(\rm iii)$ It is easy to see that $pAp\lessdot_M B$, so we may assume that $p=1_A$. 
Let $\Phi: 1_A\langle M,B\rangle1_A \to A$ be any conditional expectation such that $\rE=\Phi|_{1_A M 1_A}$ is faithful and normal. Observe that $\rE(q) \in \mathcal{Z}(A)$ and denote by $z_q$ its support projection in $\mathcal{Z}(A)$. Since $\rE(qz_{q})= \rE(q)z_{q} = \rE(q) $, it follows that $q z_{q} = q$. There is an increasing sequence $(z_n)_n$ of projections in $\mathcal{Z}(A)z_q$ such that $\rE(q)z_n$ is invertible in $\mathcal{Z}(A)z_n$ and that $z_n \to z_q$ strongly. Put $q_n=qz_n \in A' \cap 1_A M 1_A$. For each $n \in \N$, define the conditional expectation $ \Phi_n\colon  q_n\langle M,B\rangle q_n \to Aq_n$ by the formula
	$$\forall x\in q_n\langle M,B\rangle q_n, \quad \Phi_n(x)=\rE(q)^{-\frac{1}{2}} q_n \Phi(q_n x q_n)\rE(q)^{-\frac{1}{2}} q_n=\Phi(q x q)q\rE(q)^{-1} z_n.$$
It is easy to see that $\Phi_{n}|_{q_{n} M q_{n}}$ is faithful and normal. Choose $\omega\in\beta(\N)\setminus \N$ and define the ucp map $\Phi^\omega \colon  q\langle M,B\rangle q \to qMq $ by the formula
	$$\forall x \in q\langle M,B\rangle q, \quad \Phi^\omega(x)=\sigma\text{-weak} \lim_{n \to \omega}\Phi_n(q_n xq_n).$$
Then $\Phi^{\omega}$ is a norm one projection whose range is equal to $Aq$ and so $\Phi^\omega \colon  q\langle M,B\rangle q \to Aq $ is a conditional expectation. Choose a faithful normal state $\psi_A$ on $A$ and extend it to $1_AM1_A$ by $\psi=\psi_A\circ \rE$. For every $x\in M$, we have
\begin{align*}
	(\psi\circ \Phi^\omega)(qxq)
	& = \lim_{n \to \omega}(\psi\circ\Phi_n)(q_n xq_n) \\
	& = \lim_{n \to \omega}\psi\left( \rE(q x q)\rE(q)^{-1} q z_n \right) \\
	& = \lim_{n \to \omega}\psi_A\left( \rE(q x q)\rE(q)^{-1} \rE(q) z_n \right) \\
	& = \lim_{n \to \omega}\psi_A\left( \rE(q x q)z_n \right) \\
	& = \psi_A\left( \rE(q x q) z_{q} \right) = \psi_A\left( \rE(q x qz_{q}) \right) = \psi_A\left( \rE(q x q) \right) = \psi(q x q). 
\end{align*}
 Since $\psi$ is faithful and normal on $qMq$ and $\psi|_{qMq} \circ \Phi^{\omega}|_{qMq} = \psi|_{qMq}$, it follows that $\Phi^\omega|_{qMq}$ is faithful and normal. Therefore, we conclude that $Aq \lessdot_MB$.

$(\rm iv)$ Assume that $pApq \lessdot_M B$. By item $(\rm ii)$, we have $pApq\ovt \mathbf B(\ell^2) \lessdot_{M\ovt \mathbf B(\ell^2)} B\ovt \mathbf B(\ell^2)$. We may choose sequences of partial isometries $v_n \in A$ and $w_n\in A'\cap 1_AM 1_A$ such that $p_n=v_n^* v_n \leq p$ and $q_n=w_n^*w_n \leq q$ for every $n \in \N$ and that $\sum_{n \in \N}v_n v_n^* = z_p$ and $\sum_{n \in \N}w_n w_n^* = z_q$. Here $p_n$ and $q_n$ are possibly zero for some $n\in \N$. Let $(e_{i,j})_{i,j \in \N}$ be the matrix unit for $\mathbf B(\ell^2)$ and define partial isometries 
	$$V = \sum_{n\in \N}v_n \otimes e_{1,n} \in A \ovt \mathbf B(\ell^2), \quad W = \sum_{n\in \N}w_n \otimes e_{1,n} \in (A'\cap 1_A M 1_A) \ovt \mathbf B(\ell^2).$$
Observe that
	$$VV^* = z_p \otimes e_{1,1}, \quad V^*V =\sum_{n\in \N}p_n \otimes e_{n,n}, \quad WW^* = z_q \otimes e_{1,1}, \quad W^*W =\sum_{n\in \N}q_n \otimes e_{n,n}.$$ 
Then using item $(\rm iii)$, we get that 
	$$V^*V\left(A\ovt \mathbf B(\ell^2)\right) V^*V W^*W\lessdot_{M\ovt \mathbf B(\ell^2)} B\ovt \mathbf B(\ell^2)$$
and item $(\rm i)$ implies that
	$$VV^*\left(A\ovt \mathbf B(\ell^2)\right) VV^* WW^*\lessdot_{M\ovt \mathbf B(\ell^2)} B\ovt \mathbf B(\ell^2).$$
We thus obtain $Az_pz_q\ovt \C e_{1,1}\lessdot_{M\ovt \mathbf B(\ell^2)} B\ovt \mathbf B(\ell^2)$. Using the identification $M\ovt \C e_{1,1}\cong M$, we conclude that $Az_pz_q\lessdot_{M} B$.

$(\rm v)$ Using items $(\rm iii)$ and $(\rm iv)$, we may assume that all the projections $p_i$ are contained in $\mathcal{Z}(A)$ (or in $\mathcal{Z}(A'\cap 1_AM1_A)$). Since $\mathcal Z(A) \subset \mathcal Z(A' \cap 1_{A} M 1_{A})$, we may further assume that all the projections $p_{i}$ are contained in $\mathcal Z(A' \cap 1_{A} M 1_{A})$. Put $p = \bigvee_{i \in I} p_{i} \in \mathcal Z(A' \cap 1_{A}M1_{A})$.

We first prove the case when $I=\{1,2\}$. Since $p_1\vee p_2 = p_1 + (p_2 - p_1p_2)$, using item $(\rm iii)$ and up to exchanging $p_2$ with $p_2 - p_1p_2$, we may assume that $p_1p_2 = 0$. For every $i \in \{1,2\}$, let $\Phi_i \colon p_i\langle M,B \rangle p_i \to Ap_i$ be any conditional expectation such that $\Phi_i|_{p_i Mp_i}$ is faithful and normal. Define the conditional expectation $\Phi \colon 1_A\langle M,B \rangle 1_A \to Ap_1 \oplus A p_2$ by the formula
	$$\forall x\in 1_A\langle M,B\rangle 1_A, \quad \Phi(x)=\Phi_1(p_1 x p_1) \oplus \Phi_2(p_2 x p_2).$$
Since $\Phi|_{1_AM1_A}$ is faithful and normal, this shows that $Ap_1 \oplus Ap_2 \lessdot_M B$. Since the inclusion $A \subset Ap_1 \oplus Ap_2$ is with expectation, we obtain that $A\lessdot_M B$. 

We next prove the general case. Denote by $J$ the directed set of all finite subsets of $I$. For every $\mathcal F \in J$, put $p_{\mathcal F} = \bigvee_{i \in \mathcal F} p_i \in \mathcal Z(A' \cap 1_{A}M1_{A})$. Observe that $Ap_{\mathcal F}\lessdot_MB$ by the result obtained in the last paragraph and a straightforward induction. Moreover, we have that $p_{\mathcal F} \to p$ strongly as ${\mathcal F} \to \infty$. Choose a faithful state $\psi \in (1_{A} M 1_{A})_{\ast}$ such that $A \subset 1_{A} M 1_{A}$ is globally invariant under $\sigma^{\psi}$. Since $p_{\mathcal F} \in \mathcal Z(A' \cap 1_{A} M 1_{A})$, we have $\sigma_{t}^{\psi}(p_{\mathcal F}) = p_{\mathcal F}$ for every $t \in \R$ and so $Ap_{\mathcal F} \subset p_{\mathcal F} M p_{\mathcal F}$ is globally invariant under $\sigma^{\psi}$. Denote by $\rE_{\mathcal F} : p_{\mathcal F} M p_{\mathcal F} \to A p_{\mathcal F}$ the unique $\psi$-preserving conditional expectation. Since $Ap_{\mathcal F} \lessdot_{M} B$, we may choose a conditional expectation $\Phi_{\mathcal F} : p_{\mathcal F} \langle M, B\rangle p_{\mathcal F} \to A p_{\mathcal F}$ such that $\Phi_{\mathcal F}|_{p_{\mathcal F}Mp_{\mathcal F}} = \rE_{\mathcal F}$. Choose a cofinal ultrafilter $\omega$ on $J$ and define the ucp map $\Phi^{\omega} : p\langle M, B\rangle p \to pMp$ by the formula
$$\forall x \in p\langle M, B \rangle p, \quad \Phi^{\omega}(x) = \sigma\text{-weak} \lim_{\mathcal F \to \omega} \Phi_{\mathcal F}(p_{\mathcal F} x p_{\mathcal F}).$$
Then $\Phi^{\omega}$ is a norm one projection whose range is equal to $Ap$ and so $\Phi^{\omega} : p\langle M, B\rangle p \to Ap$ is a conditional expectation. Moreover, for every $x \in pMp$, we have
\begin{align*}
(\psi \circ \Phi^{\omega})(pxp) &=  \lim_{\mathcal F \to \omega} (\psi \circ \Phi_{\mathcal F})(p_{\mathcal F} x p_{\mathcal F}) \\
&=  \lim_{\mathcal F \to \omega} \psi(p_{\mathcal F} x p_{\mathcal F}) \\
&= \psi(pxp).
\end{align*} 
Since $\psi$ is faithful and normal on $pMp$ and $\psi|_{pMp} \circ \Phi^{\omega}|_{pMp} = \psi|_{pMp}$, it follows that $\Phi^{\omega}|_{pMp}$ is faithful and normal. Therefore, we conlude that $Ap \lessdot_{M} B$.
\end{proof}

\begin{lem}\label{lem-elementary-properties2}
Let $M$ be any $\sigma$-finite von Neumann algebra and $A \subset 1_A M 1_A$ and $B \subset M$ any von Neumann subalgebras with expectation. The following assertions hold:
\begin{itemize}
\item [$(\rm i)$] Assume that there exists a conditional expectation $\Phi : 1_A \langle M, B \rangle 1_A \to A$ such that $\Phi |_{1_A M 1_A}$ is normal. Then there exists a nonzero projection $q \in \mathcal{Z}(A' \cap 1_A M 1_A)$ such that $Aq \lessdot_M B$.
\item [$(\rm ii)$] Let $q \in A' \cap 1_AM1_A$ be any nonzero projection and put $P = Aq \oplus A(1_A -q)$. Then $P \subset 1_AM1_A$ is with expectation and if $A \lessdot_M B$, then $P \lessdot_M B$.
\item [$(\rm iii)$] Assume that $A' \cap 1_AM1_A  = \mathcal Z(A)$. Let $u \in \mathcal N_{1_AM1_A}(A)$ be any element and put $Q = \langle A, u\rangle$. Then $Q \subset 1_AM1_A$ is with expectation and if $A \lessdot_M B$, then $Q \lessdot_M B$.
\end{itemize}
\end{lem}

\begin{proof}
$(\rm i)$ Denote by $q \in 1_AM1_A$ the support projection of $\Phi|_{1_AM1_A}$. Then $q \in A' \cap 1_A M 1_A$ and the mapping $\Psi : q\langle M, B\rangle q \to Aq : x \mapsto \Phi(x)q$ is a conditional expectation such that $\Psi |_{q Mq}$ is faithful and normal. Therefore, $Aq \lessdot_M B$. Using Lemma \ref{lem-elementary-properties1} $(\rm iv)$, we get the conclusion.

$(\rm ii)$ The fact that $P \subset M$ is with expectation follows from \cite[Proposition 2.2]{HU15b}. Assume that $A \lessdot_M B$. Put $p_1=q$ and $p_2=1_A-q$. For every $i \in \{1, 2\}$, by Lemma \ref{lem-elementary-properties1} $(\rm iii)$ we have $Ap_{i} \lessdot_{M} B$ and so there is a conditional expectation $\Phi_i \colon p_i\langle M,B \rangle p_i \to Ap_i$ such that $\Phi_i|_{p_i Mp_i}$ is faithful and normal. Define the conditional expectation $\Phi \colon 1_A\langle M,B \rangle 1_A \to Ap_1 \oplus A p_2$ by the formula
	$$\forall x\in 1_A\langle M,B\rangle 1_A, \quad \Phi(x)=\Phi_1(p_1 x p_1) \oplus \Phi_2(p_2 x p_2).$$
Since $\Phi|_{1_AM1_A}$ is faithful and normal, this shows that $Ap_1 \oplus Ap_2 \lessdot_M B$.

$(\rm iii)$ Choose a faithful state $\psi \in (1_{A}M1_{A})_{\ast}$ such that $A \subset 1_{A} M 1_{A}$ is globally invariant under $\sigma^{\psi}$. Since $A' \cap 1_AM1_A = \mathcal Z(A)$, \cite[Lemma 4.1]{BHV15} implies that $\sigma_t^{\psi}(u)u^* \in A$ for every $t \in \R$. Then $Q$ is globally invariant under $\sigma^{\psi}$ and so there is a unique $\psi$-preserving conditional expectation $\rE_Q :  1_AM1_A \to Q$. 

Denote by $(1_{A}M1_{A}, \rL^{2}(1_{A}M1_{A}), \mathcal J, \rL^{2}(1_{A}M1_{A})^{+})$ the standard form of $1_{A}M1_{A}$.
Since both $A$ and $Q$ are globally invariant under $\sigma^{\psi}$ and since $\mathcal J=\mathcal J_\psi$, we canonically have the following inclusion of basic constructions
$$\langle 1_{A}M1_{A}, Q\rangle = (\mathcal J Q \mathcal J)'  \subset (\mathcal J  A \mathcal J)' = \langle 1_{A}M1_{A}, A\rangle \subset \mathbf B (\rL^2(1_{A}M1_{A})).$$

Assume that $A \lessdot_M B$. By Theorem \ref{thm-characterization} $(\rm iii)$, there exists a ucp map $\Theta : 1_{A}\langle M, B\rangle1_{A} \to \langle 1_{A}M1_{A}, A\rangle$ such that $\Theta(x) = x$ for every $x \in 1_{A}M1_{A}$. Choose $\omega \in \beta(\N) \setminus \N$ and define the ucp map $\Xi : \langle 1_{A}M1_{A}, A\rangle \to \langle 1_{A}M1_{A}, A\rangle$ by the formula 
$$\forall x \in \langle 1_{A}M1_{A}, A \rangle, \quad \Xi(x) = \sigma\text{-weak} \lim_{n \to \omega} \frac1n \sum_{k = 1}^n (\mathcal Ju\mathcal J)^n \, x \, (\mathcal Ju^*\mathcal J)^n.$$
Then the range of $\Xi$ is contained in $\langle 1_{A}M1_{A}, A\rangle \cap \{\mathcal Ju\mathcal J\}' = (\mathcal JA\mathcal J)' \cap \{\mathcal Ju \mathcal J\}' = (\mathcal J Q \mathcal J)' = \langle 1_{A}M1_{A}, Q\rangle$ and $\Xi(x) = x$ for every $x \in 1_{A}M1_{A}$. The composition ucp map $\Xi \circ \Theta : 1_{A}\langle M, B\rangle1_{A} \to \langle 1_{A}M1_{A}, Q\rangle$ satisfies $(\Xi \circ \Theta)(x) = x$ for every $x \in 1_{A}M1_{A}$. Thus, we obtain $Q \lessdot_M B$ by Theorem \ref{thm-characterization} $(\rm i)$.
\end{proof}

\section{Popa's intertwining theory}

Popa introduced his powerful intertwining-by-bimodules theory in the case when the ambient von Neumann algebra is tracial \cite{Po01, Po03}. The intertwining-by-bimodules theory has recently been generalized to arbitrary von Neumann algebras in \cite{HI15a, BH16}. We use throughout the following terminology introduced in \cite[Definition 4.1]{HI15a}.

\begin{df}\label{df-intertwining}
Let $M$ be any $\sigma$-finite von Neumann algebra and $A \subset 1_A M 1_A$ any $B \subset 1_B M 1_B$ any von Neumann subalgebras with expectation. We say that $A$ {\em embeds with expectation into} $B$ {\em inside} $M$ and write $A \preceq_M B$, if there exist projections $e \in A$ and $f \in B$, a nonzero partial isometry $v \in eMf$ and a unital normal $\ast$-homomorphism $\theta: eAe \to fBf$ such that the inclusion $\theta(eAe) \subset fBf$ is with expectation and $av = v\theta(a)$ for all $a \in eAe$.
\end{df}

\subsection*{Intertwining theory and finite index inclusions}

We clarify the relationship between intertwining theory and finite index inclusions \cite{Jo82, PP84, Ko85, Po95}.  Let $B \subset M$ be any inclusion of $\sigma$-finite von Neumann algebras with faithful normal conditional expectation $\rE_B : M \to B$. As explained in \cite{Ko85} and using results in \cite{Ha77a,Ha77b,Co78}, there is a unique faithful normal semifinite operator valued weight $\widehat{\rE}_B \colon \langle M,B\rangle \to M$ such that 
	$$ \frac{\rd(\varphi_B\circ \rE_B)(J^{M} \, \cdot \, J^{M})}{\rd\psi} = \frac{\rd\varphi_B(J^{B} \, \cdot \, J^{B})}{\rd(\psi\circ \widehat{\rE}_B)}$$
for some (or equivalently any) faithful states  $\varphi_B \in B_{\ast}$ and $\psi \in M_{\ast}$.

We have $\widehat{\rE}_B(xe_Bx^*)=xx^*$ for all $x\in M$. The element $\widehat{\rE}_B(1)$ belongs to $\widehat{\mathcal{Z}(M)}^+$ and is called the {\em index} of $\rE_B$. It will be denoted by $\Ind(\rE_B) = \widehat{\rE}_B(1)$. 
We say that the inclusion $B\subset M$ has {\em finite index} if there is a faithful normal conditional expectation $\rE_B\colon M\to B$ such that $\widehat{\rE}_B(1) \in \mathcal{Z}(M)$. 

Observe that the von Neumann algebra $(B'\cap M)^{\varphi_B\circ \rE_B}$ does not depend on the choice of the faithful state $\varphi_B \in B_{\ast}$. For this reason, it will be simply denoted by $(B'\cap M)^{\rE_B}$. Let $\varphi_{B} \in B_{\ast}$ be any faithful state and put $\varphi = \varphi_{B} \circ \rE_{B} \in M_{\ast}$. Let $p\in (B'\cap M)^{\rm \rE_{B}}$ be any nonzero projection and put $\varphi_{p} = \frac{\varphi(p \, \cdot \, p)}{\varphi(p)} \in (pMp)_{\ast}$ and $r_{p} =p JpJ$. Then $(pMp, r_{p}\rL^2(M), r_{p}J^{M}r_{p}, r_{p}\rL^2(M)^+)$ is the standard form of $pMp$ (see \cite[Lemma 2.6]{Ha73}).
Observe that $\rE_B(p) \in \mathcal{Z}(B)$, $J p J \in \langle M,B \rangle$, $r_p\langle M,B \rangle r_p = \langle pMp, Bp \rangle$ and $J p J e_{B} = p e_{B}$ (since $J = J_{\varphi}$ and $p \in M^{\varphi}$). 

We investigate the index for the inclusion $Bp \subset pMp$ associated with $\rE_B$. Assume that there is $\delta>0$ such that $p\delta \leq p\rE_B(p)$ (this is always satisfied when $B$ is a factor). Then the unique $\varphi_{p}$-preserving conditional expectation $\rE_{Bp} : pMp \to Bp$ is given by the formula:
	$$\forall x \in M, \quad \rE_{Bp}(pxp) = \rE_B(pxp) \, p\rE_B(p)^{-1}.$$

The Jones projections $e_{B}$ and $e_{Bp}$ satisfy the formula $e_{Bp} = \rE_B(p)^{-1}r_pe_Br_p $ (note that $r_{p} \xi_{\varphi} = p \xi_{\varphi}$ since $p \in (B' \cap M)^{\varphi}$). The following lemma provides a sufficient condition on $p \in (B' \cap M)^{\rE_{B}}$ for the index of $Bp \subset pMp$ to be finite. It is a straightforward generalization of \cite[Proposition 4.2]{Ko85} that deals with the case when both $B$ and $M$ are  factors.

\begin{lem}\label{prop for local index}
	Keep the same notation as above. Assume that there is $\delta>0$ such that $p\delta \leq p\rE_B(p)$ and $\widehat{\rE}_B(r_p)\in M$. Then for all $x\in r_p\langle M,B \rangle r_p$, we have
\begin{equation}\label{local-index:eq1}
\delta \, \widehat{\rE}_B(xx^*)r_p \leq \widehat{\rE}_{Bp}(xx^*) \leq \| \rE_B(p) \|_\infty\,\widehat{\rE}_B(xx^*)r_p.
\end{equation}
In particular,  we have
	$$\delta \, \widehat{\rE}_B(r_p)r_p\leq \mathrm{Ind}(\rE_{Bp})\leq  \|\rE_B(p) \|_\infty\,\widehat{\rE}_B(r_p)r_p$$
and so $Bp \subset pMp$ has finite index.

If moreover $\rE_B(p)\in \mathcal Z(M)$, then for all $x\in r_p\langle M,B \rangle r_p$, we have
\begin{equation}\label{local-index:eq2}
\widehat{\rE}_{Bp}(xx^*) = \rE_B(p) \, \widehat{\rE}_B(xx^*)r_p.
\end{equation}
\end{lem}

\begin{proof}
Following closely the proof of \cite[Proposition 4.2]{Ko85} and since $J p J e_B = p e_B$, in order to prove  \eqref{local-index:eq1} and \eqref{local-index:eq2}, it suffices to consider the case when $x \in r_{p}\langle M,B \rangle r_{p}$ is of the form $x = a e_{B}$ for $a\in pMp$. Then we have 
\begin{itemize}

	\item [$(\rm i)$] $\widehat{\rE}_{Bp}(xx^*)=\widehat{\rE}_{Bp}( r_pa e_B a^* r_p) = \widehat{\rE}_{Bp}( r_pa \rE_B(p)e_{Bp}  a^* r_p) = r_pa\rE_B(p)a^*r_p$,

	\item [$(\rm ii)$] $\delta \, r_paa^*r_p\leq r_pa\rE_B(p)a^*r_p \leq \|\rE_B(p)\|_\infty\, r_paa^*r_p $ and

	\item [$(\rm iii)$] $r_p aa^* r_p=aa^* r_p = \widehat{\rE}_B(ae_Ba^*) r_p=\widehat{\rE}_B(xx^*) r_p$.
\end{itemize}
A straightforward combination of $(\rm i)$, $(\rm ii)$, $(\rm iii)$ yields  \eqref{local-index:eq1}.

If moreover $\rE_B(p) \in \mathcal Z(M)$, we have $r_px \rE_B(p)x^*r_p = \rE_B(p)r_px x^*r_p$ and similar arguments yield \eqref{local-index:eq2}.
\end{proof}

\begin{lem}\label{lemma for local index}
Let $p \in (B'\cap M)^{\rE_B}$ be any projection such that $\widehat{\rE}_B(r_{p}) \in M$. Then there is an increasing sequence of projections $p_{n} \in p(B'\cap M)^{\rE_B}p$ such that $\widehat{\rE}_B(J p_n J) \in M$ and $p_{n} \to p$ strongly. 
\end{lem}

\begin{proof}
	Observe that $\widehat{\rE}_B(J p J) \in \widehat{\mathcal{Z}(M)}^+$. For every $u\in \mathcal{U}(M)$, we have 
	$$ \|upu^* \, \widehat{\rE}_B(J p J)\|_\infty = \|p\widehat{\rE}_B(J p J)\|_\infty = \| \widehat{\rE}_B(r_p)\|_\infty <+\infty.$$
Let $z\in \mathcal{Z}(M)$ be the support projection of $p$ in $M$ and recall that $z=\bigvee_{u\in\mathcal{U}(M)} upu^*$. Denote by $I$ the directed set of all finite subsets of $\mathcal U(M)$. For every $\mathcal F \in I$, put $p_{\mathcal F} = \bigvee_{u \in \mathcal F} u p u^*$. Then we have $p_{\mathcal F} \to p$ strongly as $\mathcal F \to \infty$ and for every $\mathcal F \in I$, we have
	$$  \|p_{\mathcal F} \, \widehat{\rE}_B(J p J )\|_\infty \leq \left\|\left(\sum_{u\in \mathcal{F}}upu^*\right)  \widehat{\rE}_B(Jp J) \right\|_\infty \leq |\mathcal F| \cdot \| \widehat{\rE}_B(r_p)\|_\infty <+\infty.$$
This implies that the element $z \, \widehat{\rE}_B(J p J)$ is semifinite in $\widehat M^{+}$ and thus in $\widehat{\mathcal{Z}(M)}^+$ by uniqueness of the spectral resolution (see e.g.\ \cite[Theorem 1.5]{Ha77a}). In particular, there is an increasing sequence $(z_n)_n$ of projections in $\mathcal{Z}(M)z$ such that $\|z_n \, \widehat{\rE}_B(J p J)\|_\infty < +\infty$ and  $z_n \to z$ strongly. For every $n \in \N$, letting $p_n = p z_n  \in p(B'\cap M)^{\rE_B}p$, we have 
\begin{equation*}
\|\widehat{\rE}_B(J p_n J) \|_\infty= \|z_n \, \widehat{\rE}_B(J p J)\|_\infty <+\infty
\end{equation*}
and $p_n \to p$ strongly.
\end{proof}

We can now give a natural characterization of the property $M \preceq_M B$ in terms of finite index. 

\begin{thm}\label{embedding and finite index}
	Let $B\subset M$ be any inclusion of $\sigma$-finite von Neumann algebras with faithful normal conditional expectation $\rE_B :M \to B$. Then the following conditions are equivalent:
\begin{itemize}
	\item [$(\rm i)$] $M\preceq_M B$.
	\item [$(\rm ii)$] There exists a nonzero projection $p \in (B'\cap M)^{\rE_B}$ such that the inclusion $Bp\subset pMp$ has finite index.
\end{itemize}
\end{thm}

\begin{proof}
	$(\rm ii) \Rightarrow (\rm i)$ Fix any faithful normal conditional expectation $\rE_p \colon pMp \to Bp$. Since the operator valued weight $\widehat{\rE}_p$ is bounded, the condition in \cite[Theorem 2 (ii)]{BH16} is trivially satisfied and so $pMp\preceq_{pMp}Bp$. This implies that $M \preceq_MB$.

$(\rm i) \Rightarrow (\rm ii)$ We use an idea from the proof of \cite[Theorem 3.1]{HSV16}. Fix a faithful state $\varphi_B \in B_\ast$ and put $\varphi=\varphi_B \circ \rE_B \in M_\ast$. Let $e,f,v,\theta$ witnessing the fact that $M \preceq_M B$ as in Definition \ref{df-intertwining} and recall that $av = v\theta(a)$ for all $a \in eMe$. We may clearly assume that $e=vv^*$. Using \cite[Lemma 2.1]{HU15b}, there exists a partial isometry $w \in B$ such that $w^*w = f$ and $ww^* \in B^{\varphi_B}$. Up to replacing $f \in B$ by $ww^* \in B^{\varphi_B}$, $v$ by $vw^*$, $\theta$ by $\Ad (w) \circ \theta$ and using spatiality, we may assume that $f \in B^{\varphi_B}$. Since the inclusion $\theta(eMe)\subset fBf$ is with faithful normal conditional expectation  $\rE_\theta : fBf \to \theta(eMe)$, up to replacing $\varphi_B \in B_\ast$ by the faithful state $f(\varphi_B \circ \rE_\theta)f + f^\perp \varphi_B f^\perp \in B_\ast$, we may assume that $f(\varphi_B\circ \rE_\theta)f = f\varphi_Bf$ and note that we still have $f\in B^\varphi$.  Then the modular automorphism group  $\sigma^{\varphi}$ globally preserves $\theta(eMe)' \cap fMf$. Using \cite[Lemma 2.1]{HU15b} again, there exists a partial isometry $u \in \theta(eMe)' \cap fMf$ such that $u^*u = v^*v$ and $uu^*  \in \theta(eMe)' \cap (fMf)^\varphi$. Up to replacing $v$ by $vu^*$, we may assume that $v^*v\in \theta(eMe)' \cap (fMf)^\varphi$. We are now in the situation where $e=vv^*$, $f \in B^{\varphi_B}$, $v^*v \in \theta(eMe)' \cap (fMf)^\varphi$. Define the faithful normal positive functional $\psi \in (eMe)_\ast$ by $ \psi= v  \varphi  v^*$.

We next follow the proof of \cite[Theorem 2 $(\rm i) \Rightarrow (\rm ii)$]{BH16}. Denote by $z$ the central support projection of $e$ in $M$. We may choose a sequence of partial isometries $(w_n)_n$ in $M$ such that $w_0=e$, $w_n^*w_n\leq e$ for every $n \in \N$ and $\sum_{n \in \N} w_nw_n^* =z$. Put  $v_n=w_n v$ for every $n \in \N$. Letting  $d=\sum_{n\in \N} v_n e_B v_n^*$, we have $d \in M' \cap \langle M,B \rangle$ and $\widehat{\rE}_B(d) =z \in M$.  We extend $\psi \in (eMe)_\ast$ to a faithful normal positive functional on $M$ by the formula 
	$$ \psi = \sum_{n \in \N} 2^{-n} \, v_n  \varphi   v_n^* + \varphi z^\perp .$$
Put $\widehat{\psi}= \psi \circ \widehat{\rE}_B$, $\widehat{\varphi}= \varphi \circ \widehat{\rE}_B$ and $u_t=[\rD\psi : \rD\varphi]_t$ for every $t\in \R$. Observe that $\psi \, v_nv_n^* = 2^{-n}\, v_n  \varphi    v_n^*$ and so $u_t \sigma_t^\varphi(v_n) =2^{-{\rm i}nt}\, v_n$ for every $n \in \N$ and every $t \in \R$. Since $[\rD\psi : \rD\varphi]_t=[\rD\widehat{\psi} : \rD\widehat{\varphi}]_t$ for every $t\in \R$, we have 
\begin{align*}
	\sigma_t^{\widehat\psi}(d)
	&= u_t\sigma_t^{\widehat\varphi}(d)u_t^* \\
	&= \sum_{n \in \N} u_t \sigma_t^\varphi(v_n) \, e_B \, \sigma_t^\varphi(v_n)^{*}u_t^* \\
	&= \sum_{n \in \N} (2^{-{\rm i}nt}v_n) \, e_B \, (2^{-{\rm i}nt}v_n)^* \\
	&= \sum_{n \in \N} v_n e_B v_n^* =d
\end{align*}
and so $d \in M'\cap \langle M,B\rangle^{\widehat\psi}$. Since $d \in M'\cap \langle M,B\rangle$, we also have 
$$\sigma_t^{\widehat\varphi}(d) = u_t^*\sigma_t^{\widehat\psi}(d)u_t = u_t^*du_t =d$$
and so $d \in M'\cap \langle M,B\rangle^{\widehat\varphi}$. We may choose a nonzero spectral projection $r \in M'\cap \langle M,B\rangle^{\widehat\varphi}$ of $d \in M'\cap \langle M,B\rangle^{\widehat\varphi}$ such that $\widehat{\rE}_B(r)\in M$. Put $p= J r J \in J(M' \cap \langle M,B \rangle)J =B'\cap M $ and observe that since $\widehat{\rE}_B(pJ p J) \leq \widehat{\rE}_B(J p J)$, we have  $\widehat{\rE}_B(pJ p J) \in M$. For every $t\in \R$, since $J = J_\varphi$ and $J\Delta_\varphi^{{\rm i}t} = \Delta_\varphi^{{\rm i}t} J$, we have 
	$$J pJ= \sigma_t^{\widehat{\varphi}}(J pJ)= \Delta_\varphi^{{\rm i}t}J \, p \, J \Delta_\varphi^{-{\rm i}t}= J\Delta_\varphi^{{\rm i}t} \, p \,  \Delta_\varphi^{-{\rm i}t} J= J \sigma_t^\varphi(p) J$$
and so $\sigma_t^\varphi(p)=p$. 

Thus, we have obtained a nonzero projection $p\in (B'\cap M)^\varphi$ such that $\widehat{\rE}_B(p J p J) \in M$. By multiplying by a nonzero central projection in $B$ if necessary, we may assume that $p$ satisfies $\delta p \leq \rE_B(p)p$ for some $\delta>0$. Therefore, we can apply Proposition \ref{prop for local index} and we obtain that $Bp\subset pMp$ has finite index. 
\end{proof}

\subsection*{Intertwining theory and crossed products}

Let $B$ be any $\sigma$-finite von Neumann algebra, $\Gamma$ any discrete group and $\alpha : \Gamma \curvearrowright B$ any action.  Put $M=B\rtimes \Gamma$. Denote by $\rE_B :  M \to B$ the canonical faithful normal conditional expectation and by $(u_g)_{g \in \Gamma}$ the canonical unitaries in $M$ implementing the action $\alpha$. Throughout this subsection, we assume that the action $\alpha : \Gamma \curvearrowright B$ is {\em properly outer}, that is, for every nonzero projection $p \in B$ such that $\alpha_g(p) = p$ for every $g \in \Gamma$, the reduced action $\alpha^p : \Gamma \curvearrowright pBp$ is outer. It is well known that $\alpha : \Gamma \curvearrowright B$ is properly outer if and only if $B'\cap M =\mathcal{Z}(B)$. 

We investigate the condition $M\not\preceq_MB$ in this setting. We start by proving an easy lemma.

\begin{lem}\label{lemma for recurrent}
Keep the same notation as above. The following conditions are equivalent: 
\begin{itemize}
	\item[$\rm (i)$]  $M \preceq_M B$.
	\item[$\rm (ii)$] There is a nonzero projection $p \in \mathcal{Z}(B)$ such that $\|\sum_{g\in \Gamma} \alpha_g(p) \|_\infty < +\infty$.
	\item[$\rm (iii)$] There is a nonzero projection $p \in \mathcal{Z}(B)$ such that $\|\sum_{g\in \Gamma} \alpha_g(p)p \|_\infty < +\infty$.
\end{itemize}
\end{lem}
\begin{proof}
	Let $p\in B'\cap M=\mathcal{Z}(B)$ be any projection. Observe that $1=\sum_{g\in \Gamma} u_g e_B u_g^*$ and $pe_B p =J pJ e_B J pJ$. Then we have
\begin{align*}
	\widehat{\rE}_B(J pJ)
	&=   \widehat{\rE}_B\left(J pJ  \left(\sum_{g\in \Gamma} u_g e_B u_g^* \right)J pJ\right)\\
	&=   \widehat{\rE}_B\left( \sum_{g\in \Gamma} u_g \, J pJ e_B J pJ \, u_g^* \right)\\
	&=   \widehat{\rE}_B\left( \sum_{g\in \Gamma} u_g \, p e_B p \, u_g^* \right)\\
	&=   \sum_{g\in \Gamma} u_g p u_g^* \\
	&=  \sum_{g\in \Gamma} \alpha_g(p). 
\end{align*}

$(\rm i) \Leftrightarrow (\rm ii)$ By Theorem \cite[Theorem 2 (ii)]{BH16}, the condition $M\preceq_MB$ is equivalent to the existence of a projection $J p J \in M'\cap \langle M,B\rangle= J(B'\cap M) J = J\mathcal{Z}(B)J$ such that $\| \widehat{\rE}_B(J p J) \|_\infty < +\infty$. Using the computation above, we indeed have $(\rm i) \Leftrightarrow (\rm ii)$. 

$(\rm ii) \Leftrightarrow (\rm iii)$ This is immediate by combining the above computation and Lemma \ref{lemma for local index}.
\end{proof}

The following corollary is an immediate consequence of Lemma \ref{lemma for recurrent} and provides a useful sufficient condition for having $M\not\preceq_MB$.

\begin{cor}
	Keep the same notation as above. Assume that $|\Gamma| = +\infty$ and that there is a faithful state $\psi \in \mathcal{Z}(B)_\ast$ such that  $\psi \circ \alpha_g = \psi$ for every $g \in \Gamma$. Then we have $M\not \preceq_M B$.
\end{cor}
\begin{proof}
Suppose by contradiction that $M\preceq_MB$. By Lemma \ref{lemma for recurrent}, there is a nonzero projection $p \in \mathcal{Z}(B)$ such that $\|\sum_{g\in \Gamma} \alpha_g(p) \|_\infty < +\infty$. On the other hand, applying $\psi$, we obtain 
	$$\psi\left(\sum_{g\in \Gamma} \alpha_g(p)\right) = \sum_{g\in \Gamma} \psi(\alpha_g(p))=\sum_{g\in \Gamma} \psi(p) = |\Gamma| \cdot \psi (p) =  +\infty.$$
This is a contradiction.
\end{proof}

Using Lemma \ref{lemma for recurrent}, we now provide a natural characterization of the condition $M\not \preceq_M B$. Recall that a nonsingular action $\Gamma \curvearrowright (X, \mu)$ of a countable discrete group on a standard measure space is {\em recurrent} if for every measurable subset $W \subset X$ with $\mu(W) > 0$ and for $\mu$-almost every $x \in W$, we have $|\{g \in \Gamma \mid g \cdot x \in W\} | = +\infty$.

\begin{thm}\label{thm-crossed-products}
	Keep the same notation as above. Assume that $\Gamma$ is countable and $\mathcal{Z}(B)$ has separable predual. Write $\mathcal{Z}(B)=\rL^\infty(X,\mu)$ for some standard probability space $(X,\mu)$. Then the following conditions are equivalent. 
\begin{itemize}
	\item[$(\rm i)$]  $M \not\preceq_M B$.
	\item[$(\rm ii)$] The nonsingular action $\Gamma \curvearrowright (X, \mu)$ is recurrent.
\end{itemize}
\end{thm}

\begin{proof}
$(\rm i) \Rightarrow (\rm ii)$ Assume that $(\rm ii)$ does not hold. Then there is a non-null measurable subset $W\subset X$ such that for every $x \in W$, we have $|\{ g\in \Gamma \mid g\cdot x \in W \}| < +\infty$. Denote by $I$ the set of finite subsets of $\Gamma$ and observe that $I$ is countable. We can write $W = \bigsqcup_{\mathcal F \in I} W_{\mathcal F}$ where 
	$$W_{\mathcal F}= \left\{x\in W \mid \{ g\in \Gamma \mid g^{-1}\cdot x \in W \} = \mathcal F \right\}.$$
Since $W$ is non-null, there is $\mathcal F \in I$ such that $W_{\mathcal F}$ is non-null. Put $p= \mathbf1_{W}$, $q= \mathbf 1_{W_{\mathcal{F}}}$ and observe that $ \alpha_g(p)q = \mathbf 1_{ gW \cap W_{\mathcal{F}} }$ is nonzero only for $g\in \mathcal{F}$. Thus we obtain 
	$$\sum_{g\in \Gamma} \alpha_g(q)q \leq  \sum_{g\in \Gamma} \alpha_g(p)q =  \sum_{g\in \mathcal{F}} \alpha_g(p)q \in \mathcal Z(B).$$
Lemma \ref{lemma for recurrent} implies that $M \preceq_MB$.

$(\rm ii) \Rightarrow (\rm i)$ Assume by contradiction that $(\rm ii)$ holds and $(\rm i)$ does not hold. By Lemma \ref{lemma for recurrent} $\rm (iii)$, there is a nonzero projection $p \in \mathcal{Z}(B)$ such that $\|\sum_{g\in \Gamma} \alpha_g(p)p\|_{\infty} < +\infty$. Write $p = \mathbf 1_{W}$ for some measurable subset $W \subset X$. For every $g\in \Gamma$, we put $W_{g}=\{x\in W \mid g^{-1}\cdot x \in W \} =W\cap gW$. 

We claim that for every $n \geq 1$, there are distinct elements $g_1,\ldots ,g_n \in \Gamma$ such that $\cap_{i=1}^n W_{g_i}$ is a non-null subset. The case $n=1$ is trivial since we can choose $g_1=1_\Gamma$. Suppose that the claim holds for $n-1 \geq 1$ and take elements $g_1,\ldots,g_{n-1} \in \Gamma$ such that $\cap_{i=1}^{n - 1} W_{g_i}$ is non-null. Letting $\mathcal{F}=\{g_1,\ldots,g_{n-1}\}$, we have
\begin{align*}
	W\setminus  \left(\bigcup_{g\in \Gamma\setminus \mathcal{F}}W_{g} \right) 
	&= \bigcap_{g\in \Gamma\setminus \mathcal{F}} W\setminus W_{g} \\
	&= \bigcap_{g\in \Gamma\setminus \mathcal{F}} \{ x\in W \mid g^{-1}\cdot x \not\in W \}\\ 
	&=  \{ x\in W \mid g^{-1}\cdot x \not\in W \ \text{for all $g\in \Gamma \setminus \mathcal{F}$} \}\\
	&\subset  \{ x\in W \mid |\{ g\in \Gamma \mid g\cdot x \in W \}| \leq n-1 \}.
\end{align*}
Since the nonsingular action $\Gamma \curvearrowright (X, \mu)$ is recurrent, we have that $\{ x\in W \mid |\{ g\in \Gamma \mid g\cdot x \in W \}| \leq n-1 \}$ is a null set. Therefore, $\bigcup_{g\in \Gamma\setminus \mathcal{F}}W_{g}$ has a non-null intersection with $\cap_{i=1}^{n-1} W_{g_i}$, that is, there is $g\in \Gamma\setminus \mathcal{F}$ such that $W_g \cap \left(\cap_{i=1}^{n-1} W_{g_i}\right)$ is non-null. Thus we can put $g_n = g$ and this proves our claim. 

Now fix any $n\in \N$ and let $\mathcal{G}=\{g_1,\ldots, g_n\}\subset \Gamma$ be as in the claim. Put $W_{\mathcal{G}}= \cap_{i=1}^n W_{g_i}$ and observe that $\alpha_{g_1}(p)\alpha_{g_2}(p)\cdots \alpha_{g_n}(p)p = \mathbf 1_{W_{\mathcal{G}}} \neq 0$. In particular, we have $\alpha_g(p)p \geq \mathbf 1_{W_{\mathcal{G}}}\neq0$ for every $g\in \mathcal{G}$ and so
	$$ \sum_{g\in \Gamma} \alpha_g(p)p  \geq \sum_{g\in \mathcal{G}} \alpha_g(p)p \geq \sum_{g\in \mathcal{G}} \mathbf 1_{W_{\mathcal{G}}}  =n \mathbf 1_{W_{\mathcal{G}}}.$$
This implies $ \|\sum_{g\in \Gamma} \alpha_g(p)p\|_\infty \geq n$. Since $n \geq 1$ can be arbitrarily large, this contradicts the assumption that $\|\sum_{g\in \Gamma}\alpha_g(p)p \|_{\infty} < +\infty$.
\end{proof}

\subsection*{Intertwining theory and continuous cores}

We next investigate how the intertwining theory behaves with respect to the continuous core decomposition. Let $M$ be any $\sigma$-finite von Neumann algebra and $A \subset 1_{A} M 1_{A}$ and $B \subset 1_{B} M1_{B}$ any von Neumann subalgebras with expectation. Let $\rE_A : 1_AM1_A \to A$ and $\rE_B : 1_BM1_B \to B$ be any faithful normal conditional expectations, $\varphi_A, \varphi_B \in M_\ast$ any faithful states such that $1_A \in M^{\varphi_A}$, $\varphi_A(1_A \, \cdot \, 1_A) = \varphi_A \circ \rE_A$, $1_B \in M^{\varphi_B}$ and $\varphi_B (1_B \, \cdot \, 1_B) = \varphi_B \circ \rE_B$. Simply write $\rL(\R) = \rL_{\varphi_B}(\R)$, $\core(A) = \Pi_{\varphi_B, \varphi_A}(\core_{\varphi_A}(A))$, $\core(B) = \core_{\varphi_B}(B)$, $\core(M) = \core_{\varphi_B}(M)$. 

The main question we are interested in is to investigate  under which assumptions the implication $\core(A) \preceq_{\core(M)} \core(B) \Rightarrow A \preceq_{M} B$ holds true. Firstly, we prove a positive result that generalizes \cite[Lemma 2.4]{HU15a}.

\begin{prop}\label{prop-intertwining-semifinite}
Keep the same notation as above. Assume that $A$ is semifinite or that $B$ is of type ${\rm I}$. If $A \npreceq_M B$, then $\core(A) \npreceq_{\core(M)} \core(B)$.
\end{prop}

\begin{proof}
Firstly, assume that $A$ is semifinite. We claim that the inclusion $\pi_{\varphi_B}(A) \subset \core(A)$ is with expectation. Indeed, denote by $\Tr_A$ a distinguished faithful normal semifinite trace on $A$. Observe that $\left( \pi_{\Tr_A}(A) \subset \core_{\Tr_A}(A) \right) = \left( \pi_{\Tr_A}(A) \subset \pi_{\Tr_A}(A) \ovt \rL(\R) \right)$ and so the inclusion $\pi_{\Tr_A}(A) \subset \core_{\Tr_A}(A)$ is with expectation. By Connes' Radon--Nikodym cocycle theorem \cite[Th\'eor\`eme 1.2.1]{Co72}, we have $\left( \pi_{\varphi_B}(A) \subset \core(A) \right) \cong \left( \pi_{\varphi_A}(A) \subset \core_{\varphi_A}(A) \right) \cong \left( \pi_{\Tr_A}(A) \subset \core_{\Tr_A}(A) \right)$ and so the inclusion $\pi_{\varphi_B}(A) \subset \core(A)$ is with expectation. 

Let $e \in A$ be any nonzero finite projection with central support in $A$ equal to $1_{A}$. Since $A \npreceq_{M} B$, we have $eAe \npreceq_{M} B$. Since $e$ is a finite projection in $A$, $eAe$ is a finite von Neumann algebra and \cite[Theorem 4.3]{HI15a} yields a net of unitaries $u_j \in \mathcal U(eAe)$ such that $\rE_B(x^* u_j y) \to 0$ $\ast$-strongly as $j \to \infty$ for all $x, y \in eM1_B$. By construction of $\core(M)$, this implies that $$\rE_{\core(B)}(X^*\pi_{\varphi_B}(u_j)Y) \to 0 \quad \text{strongly as} \quad j \to \infty$$ 
for all $X, Y \in \spn \left\{\pi_{\varphi_B}(e) \, \pi_{\varphi_B}(z) \lambda_{\varphi_{B}}(t) \, \pi_{\varphi_B}(1_B) : z \in M, t \in \R \right\}$. By $\sigma$-weak density and since $\pi_{\varphi_B}(eAe) \subset \pi_{\varphi_B}(e)\core(M)\pi_{\varphi_B}(e)$ is a finite von Neumann subalgebra with expectation, \cite[Theorem 4.3 $(1) \Rightarrow (5)$]{HI15a} implies that $\pi_{\varphi_B}(eAe) \npreceq_{\core(M)} \core(B)$. Since the central support of $e$ in $A$ is equal to $1_{A}$, we have $\pi_{\varphi_B}(A) \npreceq_{\core(M)} \core(B)$ by \cite[Remark 4.2 (4)]{HI15a}. Since $\pi_{\varphi_B}(A) \subset \core(A)$ is with expectation, we finally have $\core(A) \npreceq_{\core(M)} \core(B)$ by \cite[Lemma 4.8]{HI15a}.

Secondly, asssume that $B$ is of type ${\rm I}$. Denote by $z \in \mathcal Z(A)$ the unique central projection such that $Az$ is semifinite and $Az^{\perp}$ is of type ${\rm III}$. Since $\core(Az^{\perp})$ is of type ${\rm II}$ and $\core(B)$ is of type ${\rm I}$, we have $\core(Az^{\perp}) \npreceq_{\core(M)} \core(B)$. Since $A \npreceq_{M} B$, we have $Az \npreceq_{M} B$. Since $Az$ is semifinite, we have $\core(Az) \npreceq_{\core(M)} \core(B)$ by the result in the previous paragraph. Since $\core(A) = \core(Az) \oplus \core(Az^{\perp})$, we obtain $\core(A) \npreceq_{\core(M)} \core(B)$.
\end{proof}

Secondly, we prove a negative result that demonstrates the subtle difference between intertwining of the continuous cores and intertwining of the original algebras. We refer to \cite[Th\'eor\`eme 4.4.1]{Co72} for the {\em discrete decomposition} of type ${\rm III}_{\lambda}$ factors where $\lambda \in (0, 1)$.

\begin{thm}\label{rem-intertwining}
Let $\lambda \in (0, 1)$ and $M$ be any $\sigma$-finite type ${\rm III_{\lambda}}$ factor. Write $M = B \rtimes \Z$ for the discrete decomposition of $M$ where $B$ is a type ${\rm II_{\infty}}$ factor and $\Z \curvearrowright B$ is a $\lambda$-trace scaling action.

Then $M \npreceq_{M} B$ but there is a nonzero projection $p \in \mathcal Z(\core(B))$ such that $\core(B) p = p\core(M)p$. In particular, we have $\core(M) \preceq_{\core(M)} \core(B)$.
\end{thm}

\begin{proof}
Observe that since $\Z \curvearrowright B$ is $\lambda$-trace scaling and thus outer, we have $B' \cap M = \C 1$. Then there is a unique faithful normal conditional expectation $\rE_{B} : M \to B$ (see \cite[Th\'eor\`eme 1.5.5]{Co72}). It follows that the inclusion $\core(B) \subset \core(M)$ is canonical and we moreover have $\core(M) = \core(B) \rtimes \Z$. Fix a faithful normal semifinite trace $\Tr_{B}$ on $B$ and denote by $\psi = \Tr_{B} \circ \rE_{B}$ the corresponding {\em generalized trace} on $M$ (in the sense of \cite[D\'efinition 4.3.1]{Co72}). Then we have $\left( \core(B) \subset \core(M)\right) \cong \left( \core_{\psi}(B) \subset \core_{\psi}(M)\right)$, $\core(B) = B \ovt \rL(\R)$ and $\mathcal Z(\core(B)) = \rL(\R)$. 

Denote by $(u^{n})_{n \in \Z}$ the canonical unitaries in $M$ implementing the action $\Z \curvearrowright B$ and by $(v_{t})_{t \in \R}$ the canonical unitaries in $\core(M)$ implementing the modular action $\sigma^{\psi} : \R \curvearrowright M$. For every $n \in \Z$ and every $t \in \R$, we have $\sigma_{t}^{\psi}(u^{n}) =  \lambda^{{\rm i} n t} u^{n}$ and so $u^{n} v_{t} u^{-n} = \lambda^{-{\rm i} n t} v_{t}$. 
By Fourier transform, we have $\rL^{\infty}(\R) \cong \rL(\R)$ and the action $\alpha : \Z \curvearrowright \mathcal Z(\core(B))$ is simply given by the translation action $\Z \curvearrowright \R$ where $n \cdot x = x + n\log(\lambda)$. Denote by $p \in \rL(\R)$ the orthogonal projection corresponding to the interval $I = [0, \log(\lambda)[$. For every $n \in \Z \setminus \{0\}$, we have $n \cdot I \cap I = \emptyset$ and thus $\alpha_{n}(p)p = 0$. 

Put $r_{p} = p JpJ \in \langle \core(M), \core(B)\rangle$. Then we have $r_{p} \langle \core(M), \core(B)\rangle r_{p} = \langle p\core(M)p, \core(B)p\rangle$. The computation in Lemma \ref{lemma for recurrent} shows that $\widehat\rE_{\core(B)}(r_{p}) = \sum_{n \in \Z} \alpha_{n}(p) p = p$. Then \eqref{local-index:eq1} shows that $\widehat\rE_{\core(B)p}(r_{p}) = \Ind(\widehat\rE_{\core(B)p}) = \widehat \rE_{\core(B)}(r_{p})r_{p} = r_{p}$. This implies that $\core(B)p = p \core(M)p$.
\end{proof}

\subsection*{Intertwining theory and relative amenability}

We next prove a useful fact that relates intertwining theory and relative amenability.

\begin{prop}\label{lem-intertwining-relative-amenability}
Let $M$ be any $\sigma$-finite von Neumann algebra and $A \subset 1_A M 1_A$ and $B \subset M$ any von Neumann subalgebras with expectation. If $A \preceq_M B$, then there exists a nonzero projection $r \in \mathcal Z(A' \cap 1_A M 1_A)$ such that $Ar \lessdot_M B$.
\end{prop}

\begin{proof}
By \cite[Theorem 2]{BH16}, there exists a nonzero element $d \in (A' \cap 1_A\langle M, B\rangle1_A)^+$ such that $\rT_M(d) \in M$. Then there exists $\varepsilon > 0$ such that $p = \mathbf 1_{[\varepsilon, +\infty[}(d) \in A' \cap 1_A \langle M, B\rangle 1_A$ is a nonzero projection. Since $\varepsilon p \leq d p \leq d$, we have $\rT_M(p) \leq \varepsilon^{-1}\rT_M(dp) \leq \varepsilon^{-1}\rT_M(d)$ and hence $\rT_M(p) \in M$. Denote by $z \in \mathcal Z(A)$ the unique nonzero projection such that $A(1_A - z) = \ker(A \to Ap : a \mapsto ap)$. Observe that $zp = p$ and the map $\iota : Az \to Ap : az \mapsto azp$ is a unital normal $\ast$-isomorphism. Define the nonzero spectral projection $q = \mathbf 1_{[\frac12 \|\rE_A(\rT_M(p))\|_\infty, \|\rE_A(\rT_M(p))\|_\infty]}(\rE_A(\rT_M(p))) \in \mathcal Z(Az)$ and put $c = (\rE_A(\rT_M(p)) \, q)^{-1/2} \in \mathcal Z(Az)^+$. Then we have $c \, \rE_A(\rT_M(p)) \, c = q$.

Define the map $\Phi : q \langle M, B\rangle q \to A q$ by the formula $\Phi(x) = \iota^{-1}(c \, \rE_A(\rT_M(p x p)) \, c \; p)$ for every $x \in q \langle M, B\rangle q$. Then $\Phi$ is normal and for every $a \in A$, we have
$$
\Phi(aq) = \iota^{-1}(c \, \rE_A(\rT_M(p \, aq \, p)) c \; p) = \iota^{-1}(aq \; c \, \rE_A(\rT_M(p)) \, c \; p) = \iota^{-1}(aq \; p) = aq.
$$
Then $\Phi : q \langle M, B\rangle q \to A q$ is a normal conditional expectation. By Lemma \ref{lem-elementary-properties2} $(\rm i)$, there exists a nonzero projection $r \in (Aq)' \cap qMq \subset A' \cap 1_A M 1_A$ such that $Ar \lessdot_M B$. Using Lemma \ref{lem-elementary-properties1} $(\rm iv)$, we get the conclusion.
\end{proof}

\subsection*{Intertwining theory and ultraproducts}

In this subsection, we investigate how the intertwining theory behaves with respect to ultraproducts. The main result of this subsection, Theorem \ref{thm-intertwining-ultraproduct} below, will be a crucial tool in the proof of Theorems \ref{thmC} and \ref{thmE}. 

\begin{thm}\label{thm-intertwining-ultraproduct}
Let $B \subset M$ be any inclusion of $\sigma$-finite von Neumann algebras with expectation. Let $J$ be any nonempty directed set and $\omega$ any cofinal ultrafilter on $J$. Put $Q = M \vee B^\omega \subset M^\omega$. Assume that at least one of the following conditions is satisfied:
\begin{itemize}
\item [$(\rm i)$] The inclusion $B \subset M$ is regular and $M \npreceq_M B$.
\item [$(\rm ii)$] $B$ is of type ${\rm I}$ and $M \npreceq_{M} B$.
\item [$(\rm iii)$] $B$ is semifinite and $M$ is of type ${\rm III}$.
\end{itemize}
Then in each case, we have $M \npreceq_Q B^\omega$.
\end{thm}

Before going into the proof of Theorem \ref{thm-intertwining-ultraproduct}, we need to prove a few technical lemmas.

\begin{lem}\label{lemma3 for type I}
	Let $M$ be any $\sigma$-finite von Neumann algebra and $A \subset 1_A M 1_A$ and $B \subset 1_{B}M1_{B}$ any von Neumann subalgebras with expectation. Let $(z_i)_{i \in I}$ (resp.\ $(q_j)_{j \in J}$) be any family of nonzero pairwise orthogonal projections in $\mathcal{Z}(A)$ (resp.\ in $B'\cap 1_B M 1_B$) such that $\sum_{i \in I} z_i =1_A$ (resp.\ $\sum_{j \in J} q_j = 1_B$). 

Then we have that $A\preceq_M B$ if and only if there exists $(i, j) \in I \times J$ such that $Az_i \preceq_M Bq_j$.
\end{lem}

\begin{proof}
The proof is exactly the same as the one in \cite[Remark 4.2 (2)]{HI15a} where the case $q_j \in \mathcal{Z}(B)$ is considered.
\end{proof}

\begin{lem}\label{lemma1 for type I}
	Let $M$ be any $\sigma$-finite von Neumann algebra and $A \subset 1_A M 1_A$ and $B \subset 1_{B}M1_{B}$ any von Neumann subalgebras with expectation. The following assertions hold true.
\begin{itemize}
	\item[$\rm (i)$]  Assume that $A$ is finite and of type ${\rm I}$. Then for every maximal abelian subalgebra $C\subset A$, we have that $A\preceq_M B$ if and only if $C \preceq_M B$.
	\item[$\rm (ii)$] Assume that $B$ is finite and of type ${\rm I}$. Then for every maximal abelian subalgebra $C\subset B$, we have that $A\preceq_M B$ if and only if $A \preceq_M C$.
\end{itemize}
\end{lem}

\begin{proof}
$(\rm i)$ Write $A=\bigoplus_{n \geq 1} A_n \ovt \mathbf M_n(\C)$ where $A_n$ is an abelian von Neumann algebra for every $n \geq 1$. Using \cite[Theorem 3.19]{Ka82} and up to conjugating by a unitary in $A$, we may assume that $C=\bigoplus_{n \geq1} A_n \ovt \C^n$. 
For every $n \geq 1$, denote by $(e_{i,j}^n)_{i,j=1}^n$ the matrix unit of $\mathbf M_n(\C)$. By Lemma \ref{lemma3 for type I}, we have that $A \preceq_M B$ if and only if $A_n \ovt \mathbf M_n \preceq_M B$ for some $n\geq 1$, which is equivalent to $A_n\ovt \C e^n_{i,i} \preceq_M B$ for some $n \geq 1$ and $i\in \{1,\ldots,n\}$  by \cite[Remark 4.2 (4)]{HI15a}. Observe that all such projections $1_{A_n}\otimes e_{i,i}^n$ are in $C$ and $C(1_{A_n}\otimes e_{i,i}^n)=A_n \ovt \C e_{i,i}^n $. Since $\sum_{n\geq 1}\sum_{i=1}^n 1_{A_n}\otimes e_{i,i}^n =1_A$, again by Lemma \ref{lemma3 for type I}, this condition is equivalent to $C\preceq_M B$.

$(\rm ii)$ Using \cite[Remark 4.5]{HI15a} and Lemma \ref{lemma3 for type I}, the exact same argument works for the case when $B$ is finite and of type ${\rm I}$.
\end{proof}

\begin{lem}\label{lemma2 for type I}
Let $N \subset M$ be any inclusion of abelian $\sigma$-finite von Neumann algebras and $A \subset N 1_A$ and $B \subset N1_{B}$ any von Neumann subalgebras. If $A\not\preceq_N B$, then $A\not\preceq_{M}B$.
\end{lem}

\begin{proof}
Assume that $A \preceq_{M}B$. Then there exist projections $e \in A$, $f \in B$, a nonzero partial isometry $v \in e M f$ and a unital normal $\ast$-homomorphism $\theta : Ae \to Bf$ such that $av = v \theta(ae)$ for every $a \in A$. Observe that $0 \neq vv^{*} = v^{*}v \leq ef$. For every $a \in A$, we have $ae \,  v = v \, \theta(a e) = \theta(ae) \, v$ and so $ae \, vv^{*} = \theta(ae) \, vv^{*}$. Denote by $z \in N$ the unique projection such that $\ker (N \to N vv^{*} : x \mapsto x \, vv^{*}) = Nz^{\perp}$. Since $z \, vv^{*} = vv^{*}$, we have $vv^{*} \leq z$. For every $a \in A$, since $ae , \theta(ae) \in N$, using the $\ast$-isomorphism $Nvv^{*} \to Nz : yvv^{*} \mapsto yz$, we have $ae \,  z = \theta(ae) \, z = z \, \theta(ae)$ and so $ae \,  ezf = ezf \, \theta(ae)$. Since $ezf \neq 0$, we obtain $A\preceq_N B$.
\end{proof}

\begin{proof}[Proof of Theorem \ref{thm-intertwining-ultraproduct}]

$(\rm i)$ Choose a faithful normal conditional expectation $\rE_B : M \to B$ and a faithful state $\varphi \in M_\ast$ such that $\varphi = \varphi \circ \rE_B$. Then all the von Neumann subalgebras $M, B^\omega, Q \subset M^\omega$ are globally invariant under $\sigma^{\varphi^\omega}$ and thus they are all with expectation. Assume that $M \preceq_Q B^\omega$. By \cite[Theorem 2]{BH16}, there exists a nonzero normal $M$-$M$-bimodular completely positive map $\Phi : \langle Q, B^\omega \rangle \to M$. For every $x \in M$, we have $\rE_{B^\omega}(x) = \rE_B(x)$. It follows that the map $\Theta : \langle M, B \rangle \to \langle Q, B^\omega \rangle : xe_B y \mapsto x e_{B^\omega} y$ extends to a well defined normal $\ast$-homomorphism that is moreover $M$-$M$-bimodular. Observe that $\mathcal N_M(B) \subset \mathcal N_{Q}(B^\omega)$. Since $B \subset M$ is regular, we have $\mathcal N_M(B)\dpr = M$ and thus the $\rL^2$-closure of $M\cdot  \rL^2(B^\omega)$ is equal to $\rL^2(Q)$. This implies that the map $\Theta$ is unital. Then $\Phi \circ \Theta : \langle M, B\rangle \to M$ is a nonzero normal $M$-$M$-bimodular completely positive map and thus $M \preceq_M B$ by \cite[Theorem 2]{BH16}.

$(\rm ii)$ The proof of this item is the most involved. Assume that $M \npreceq_{M} B$. We will actually prove the stronger fact that  $M  \npreceq_{M^\omega} B^\omega$. 

Denote by $z\in \mathcal{Z}(M)$ the unique central projection such that $Mz$ is of type ${\rm I}$ and that $M(1-z)$ has no nonzero type ${\rm I}$ direct summand. Then $M(1-z) \not\preceq_{M^\omega} B^\omega$ trivially holds. By Lemma \ref{lemma3 for type I}, in order to prove that $M \not\preceq_{M^\omega} B^\omega $, it suffices to show that $Mz \not\preceq_{M^\omega} B^\omega$, which is equivalent to $Mz \not\preceq_{(Mz)^\omega} (Bz)^\omega $ since $z \in \mathcal Z(M) \subset \mathcal Z(M^{\omega})$. Without loss of generality, we may assume that $z=1$ and $M$ is of type ${\rm I}$.

	Put $\N^{*} = \N \setminus \{0\}$. Since $M$ is of type ${\rm I}$, there is a family of nonzero pairwise orthogonal projections $(z_n)_{n\in S}$ in $\mathcal{Z}(M)$ (with $S \subset \N^{*} \cup \{\infty\}$) such that $Mz_n$ is of type ${\rm I}_n$ for all $n\in S$. By Lemma \ref{lemma3 for type I}, we have that $Mz_n \not\preceq_{M}B$ for every $n\in S$ and so $Mz_n \not\preceq_{Mz_n}Bz_n$ for every $n\in S$. We will show that $Mz_n \not\preceq_{(Mz_{n})^\omega} (Bz_{n})^\omega$ for every $n\in S$. Since $z_{n} \in \mathcal Z(M) \subset \mathcal Z(M^{\omega})$, this will imply that $Mz_n \not\preceq_{M^\omega} B^\omega$ for every $n\in S$ and we will obtain $M\not\preceq_{M^\omega} B^\omega$ by Lemma \ref{lemma3 for type I}. Thus, we may assume from now on that $M$ is of type ${\rm I}_n$ for some $n\in\N^{*} \cup \{\infty\}$.

{\bf Assume first that $M$ is of type ${\rm I}_n$ for some $n\in \N^{*}$.} Write $M = A\ovt \mathbf M_n=\bigoplus_{i,j=1}^n A\ovt \C e_{i,j}$ where $A \otimes \C 1_{n} \subset M$ is an abelian von Neumann subalgebra and $(e_{i,j})_{i,j}$ is the matrix unit in $\mathbf M_{n}(\C)$. Let $C_B \subset B$ be a maximal abelian subalgebra and choose a maximal abelian subalgebra $C_M\subset M$ that contains $C_B$. Since $M \npreceq_{M} B$, we have $C_M \not\preceq_M C_B$ by Lemma \ref{lemma1 for type I} and so $C_M \not\preceq_{C_M} C_B$. Using \cite[Theorem 3.19]{Ka82} and up to conjugating by a unitary $u \in M$ and replacing $B$ by $u B u^{*}$ (we still have $M \npreceq_{M} uBu^{*}$), we may assume that $C_M = A \ovt \C^{n} = \bigoplus_{i=1}^n A\ovt \C e_{i,i}$. Since $C_B \subset C_M$, we have $C_B \subset \sum_{i=1}^n B_i \ovt \C e_{i,i}$ where $B_i\subset A$ is a unital von Neumann subalgebra given by $B_i \ovt \C e_{i, i} =  C_{B} (1_{A} \otimes e_{i,i}) $ for each $1\leq i \leq n$. We claim that $A \not\preceq_{A} B_i$ for every $1\leq i \leq n$. Indeed, for every $1\leq i \leq n$, put $q_i=1_A \otimes e_{i,i}$ and observe that $q_{i} \in C_M$, $q_{i} \in C_B' \cap C_M$ and $\sum_{i=1}^n q_i = 1$. Using Lemma \ref{lemma3 for type I}, since $C_M \not\preceq_{C_M} C_B$, we have $C_M q_i \not\preceq_{C_M} C_B q_i$ and so $C_M q_i \not\preceq_{C_M q_i} C_B q_i$ for every $1\leq i \leq n$. Since $C_Mq_i = A\ovt \C e_{i,i}$ and $C_B q_i = B_i \ovt \C e_{i,i}$, we indeed have $A \not\preceq_{A} B_i$ for every $1\leq i \leq n$.

Since $A$ is abelian, the inclusion $B_i \subset A$ is regular. By item $(\rm i)$, we have  $A\not \preceq_{A\vee B_i^\omega} B_i^\omega$ for every $1\leq i \leq n$. Since $A\vee B_i^\omega \subset A^{\omega}$, by Lemma \ref{lemma2 for type I}, we obtain $A\not \preceq_{A^\omega} B_i^\omega$ for every $1\leq i \leq n$. Observe that for every $1 \leq i \leq n$, we have
	$$C_Mq_i=A\ovt \C e_{i,i}, \quad q_i M^\omega q_i = A^\omega \ovt \C e_{i,i}, \quad  C_B^\omega q_i=B_i^\omega \ovt \C e_{i,i}.$$
This implies that $C_M q_i \npreceq_{q_i M^\omega q_i} C_B^\omega q_i$ and so $C_M q_i \npreceq_{M^\omega} C_B^\omega q_i$ for every $1 \leq i \leq n$. Put $v_{i, j} = 1_{A} \otimes e_{i, j}$ for all $1 \leq i, j \leq n$. Since $v_{i, j}^{*}v_{i, j} = q_{j}$, $v_{i, j}v_{i, j}^{*} = q_{i}$ and $v_{i, j}^{*} \, C_{M}q_{i} \, v_{i, j} = C_{M}q_{j}$, this further implies that $C_M q_i \npreceq_{M^\omega} C_B^\omega q_j$ for all $1 \leq i, j \leq n$. By Lemma \ref{lemma3 for type I}, we obtain $C_M \npreceq_{M^\omega} C_B^\omega$. Since $C_B^\omega \subset B^\omega$ is maximal abelian (see e.g.\ \cite[Theorem A.1.2]{Po95}), we finally obtain $M \not \preceq_{M^\omega} B^\omega$ by Lemma \ref{lemma1 for type I}.

{\bf Assume now that $M$ is of type ${\rm I}_\infty$.} We first study the case when $B$ is finite. Since $B\subset M$ is with expectation, there is a family of nonzero pairwise orthogonal projections $(q_j)_{j \in J}$ in $B'\cap M$ such that $\sum_{j \in} q_j=1$ and each $q_j$ is finite in $M$. Indeed, observe that any semifinite trace on $M$ is still semifinite on $B'\cap M$ (see e.g.\ \cite[Theorem 4.3 $(1) \Rightarrow (2)$]{HI15a}.) By contradiction, assume that $M \preceq_{M^\omega} B^\omega$. By Lemma \ref{lemma3 for type I}, there exists $j \in J$ such that $M \preceq_{M^\omega} B^\omega q_j$. Letting $q=q_j$ and writing $z$ for the central support projection of $q$ in $M$, we have $Mz \preceq_{M^\omega z} B^\omega q$. By \cite[Remark 4.2 (4)]{HI15a}, we have $qMq \preceq_{M^\omega z} B^\omega q$ and so $qMq \preceq_{q M^\omega q} B^\omega q$, that is, $qMq \preceq_{(qMq)^\omega} (Bq)^\omega$. Since $qMq$ is finite, the result we obtained above shows that $qMq \preceq_{qMq} Bq$ and so $M \preceq_M B$. This is a contradiction.

We study now the case when $B$ is semifinite. Let $q\in B$ be any finite projection with central support equal to $1$. Since $M \not \preceq_M B$, we have $qMq \not \preceq_{qMq} qBq$. Using the result in the last paragraph, we obtain $qMq \not \preceq_{(qMq)^\omega} (qBq)^\omega$ and so $qMq \not \preceq_{M^\omega} qB^\omega q$. Since $q$ has the central support equal $1$ in both $M$ and $B^\omega$, we finally obtain $M \not \preceq_{M^\omega} B^\omega $ by \cite[Remarks 4.2 (4) and 4.5]{HI15a}.

$(\rm iii)$ If $B$ is semifinite then $B^\omega$ is semifinite as well. Since $M$ is of type ${\rm III}$ and $B^\omega$ is semifinite, we necessarily have $M \npreceq_{Q} B^\omega$.
\end{proof}

\section{Proofs of the main results}\label{section:results}

\subsection*{Amalgamated free product von Neumann algebras}

For each $i \in \{ 1, 2 \}$, let $B \subset M_i$ be any inclusion of $\sigma$-finite von Neumann algebras with faithful normal conditional expectation $\rE_i : M_i \to B$. The {\em amalgamated free product} $(M, \rE) = (M_1, \rE_1) \ast_B (M_2, \rE_2)$ is a pair of von Neumann algebra $M$ generated by $M_1$ and $M_2$ and faithful normal conditional expectation $\rE : M \to B$ such that $M_1, M_2$ are {\em freely independent} with respect to $\rE$:
$$\rE(x_1 \cdots x_n) = 0 \; \text{ whenever } \; n \geq 1, \; x_j \in M_{i_j}^\circ \; \text{ and } \; i_1 \neq \cdots \neq  i_{n}.$$
Here and in what follows, we denote by $M_i^\circ = \ker(\rE_i)$. We refer to the product $x_1 \cdots x_n$ where $x_j \in M_{i_j}^\circ$ and $i_1 \neq \cdots \neq i_{n}$ as a {\em reduced word} in $M_{i_1}^\circ \cdots M_{i_n}^\circ$ of {\em length} $n \geq 1$. The linear span of $B$ and of all the reduced words in $M_{i_1}^\circ \cdots M_{i_n}^\circ$ where $n \geq 1$ and $i_1 \neq \cdots \neq i_{n}$ forms a unital strongly dense $\ast$-subalgebra of $M$. We call the resulting $M$ the \emph{amalgamated free product von Neumann algebra} of $(M_1,\rE_1)$ and $(M_2,\rE_2)$. We simply write AFP for amalgamated free product.

When $B$ is a semifinite von Neumann algebra with faithful normal semifinite trace $\Tr$ and the weight $\Tr \circ \rE_i$ is tracial on $M_i$ for every $i \in \{1, 2\}$, then the weight $\Tr \circ \rE$ is  tracial on $M$ (see \cite[Proposition 3.1]{Po90} for the finite case and \cite[Theorem 2.6]{Ue98} for the general case). In particular, $M$ is a semifinite von Neumann algebra. In that case, we refer to $(M, \rE) = (M_1, \rE_1) \ast_B (M_2, \rE_2)$ as a {\em semifinite} AFP von Neumann algebra.

Let $\varphi \in B_\ast$ be any faithful state. Then for all $t \in \R$, we have $\sigma_t^{\varphi \circ \rE} = \sigma_t^{\varphi \circ \rE_1} \ast_B \sigma_t^{\varphi \circ \rE_2}$ (see \cite[Theorem 2.6]{Ue98}). By \cite[Theorem IX.4.2]{Ta03}, there exists a unique $\varphi\circ \rE$-preserving conditional expectation $\rE_{M_1} : M \to M_1$. Moreover, we have $\rE_{M_1}(x_1 \cdots x_n) = 0$ for all the reduced words $x_1 \cdots x_n$ that contain at least one letter from $M_2^\circ$ (see e.g.\ \cite[Lemma 2.1]{Ue10}). We denote by $M \ominus M_1 = \ker(\rE_{M_1})$. For more on AFP von Neumann algebras, we refer the reader to \cite{BHR12, Po90, Ue98, Ue10, Ue12, Vo85, VDN92}.

For the rest of this section, we fix an inclusion of $\sigma$-finite von Neumann algebras $B \subset M_i$ with faithful normal conditional expectation $\rE_i : M_i \to B$ for every $i \in \{1, 2\}$. We then denote by $(M, \rE) = (M_1, \rE_1) \ast_B (M_2, \rE_2)$ (or simply by $M = M_{1} \ast_{B} M_{2}$) the corresponding AFP von Neumann algebra.

We first recall a key result about amenable absorption in arbitrary AFP von Neumann algebras due to Boutonnet--Houdayer \cite[Main Theorem]{BH16}. We state the following slightly more general version for our purposes.

\begin{thm}[{\cite[Main Theorem]{BH16}}]\label{thm-maximal-amenability}
Keep the same notation as above. Let $Q \subset M$ be any von Neumann subalgebra with expectation. Assume that the following conditions are satisfied:
\begin{itemize}
\item [$(\rm i)$] $Q \lessdot_M M_1$
\item [$(\rm ii)$] There exists a von Neumann subalgebra $A \subset Q \cap M_1$ that is with expectation in $M_1$ such that $A \npreceq_{M_1} B$. 
\end{itemize}
Then we have $Q \subset M_1$.
\end{thm} 

\begin{proof}
The proof of \cite[Main Theorem]{BH16} is exactly the same after replacing $Q \cap M_1$ by $A$.
\end{proof}

We next prove a technical result, partially generalizing \cite[Theorem 1.1]{IPP05}, about controlling relative commutants and normalizers of von Neumann subalgebras with expectation inside arbitrary AFP von Neumann algebras.

\begin{thm}\label{thm-control}
Keep the same notation as above. Assume that $M_1 \npreceq_{M_1} B$. Let $p \in M_1$ be any nonzero projection and $A \subset pM_1p$ any von Neumann subalgebra with expectation. If $A \npreceq_{M_1} B$, then $\mathcal N_{pMp}(A)\dpr \subset pM_1p$.
\end{thm}

\begin{proof}
First, we observe that we may assume that $p = 1$. Indeed, we may choose a faithful state $\psi \in M_\ast$ such that $p \in M^\psi$ and such that $M_1 \subset M$ and $A \subset pM_1p$ are both globally invariant under $\sigma^\psi$. Then $\mathcal A = A \oplus p^\perp M_1 p^\perp$ is globally invariant under $\sigma^\psi$ and thus $\mathcal A \subset M_1$ is with expectation. Since $A \npreceq_{M_1} B$ and $M_1 \npreceq_{M_1} B$, we have $\mathcal A \npreceq_{M_1} B$. We moreover have $\mathcal N_{pMp}(A) \dpr \oplus p^\perp M_1 p^\perp \subset \mathcal N_{M}(\mathcal A)\dpr$. Thus, without loss of generality, we may assume that $p = 1$.

Next, we prove that $A' \cap M \subset M_1$. Let $z \in A' \cap M$ be any projection and put $P = Az \oplus Az^\perp$. Since $A \subset M_1$ is with expectation, we have $A \lessdot_M M_1$. By Lemma \ref{lem-elementary-properties2} $(\rm ii)$, $P \subset M$ is with expectation and $P \lessdot_M M_1$. Since $A \subset P \cap M_1$ is with expectation in $M_1$ and since $A \npreceq_{M_1} B$, Theorem \ref{thm-maximal-amenability} implies that $P \subset M_1$ and so $z \in M_1$. Since this holds true for every projection $z \in A' \cap M$, we obtain $A' \cap M \subset M_1$.

Finally, we prove that $\mathcal N_M(A)\dpr \subset M_1$. We observe that we may assume that $A' \cap M = \mathcal Z(A)$. Indeed, put $\mathcal A = A \vee (A' \cap M)$. By the previous paragraph, we have $\mathcal A \subset M_{1}$ and $\mathcal A \subset M_{1}$ is with expectation. Moreover, we have $\mathcal A' \cap M = \mathcal Z(\mathcal A)$, $\mathcal A \npreceq_{M_1} B$ and $\mathcal N_M(A)\dpr \subset \mathcal N_M(\mathcal A)\dpr$. Thus, without loss of generality, we may assume that $A \subset M_1$ is with expectation, $A' \cap M = \mathcal Z(A)$ and $A \npreceq_{M_1} B$. Let $u \in \mathcal N_M(A)$ be any element and put $Q = \langle A, u\rangle$. Since $A \subset M_1$ is with expectation, we have $A \lessdot_M M_1$. By Lemma \ref{lem-elementary-properties2} $(\rm iii)$, $Q \subset M$ is with expectation and $Q \lessdot_M M_1$. Since $A \subset Q \cap M_1$ is with expectation in $M_1$ and since $A \npreceq_{M_1} B$, Theorem \ref{thm-maximal-amenability} implies that $Q \subset M_1$ and so $u \in M_1$. Since this holds true for every unitary $u \in \mathcal N_M(A)$, we obtain $\mathcal N_M(A)\dpr \subset M_1$.
\end{proof}

\subsection*{Proof of Theorem \ref{thmA}} Keep the same notation as above. We first prove the following technical result.

\begin{claim}\label{claim-technical}
If the inclusion $B \subset M_{2}$ is entirely nontrivial, then so is $\core_{\varphi}(B) \subset \core_{\varphi}(M_2)$.
\end{claim}

\begin{proof}[Proof of Claim \ref{claim-technical}]
We prove the contrapositive. Assume that there exists a nonzero projection $z \in \mathcal Z(\core_{\varphi}(B) ) \cap \mathcal Z(\core_{\varphi}(M_2))$ such that $\core_{\varphi}(B)  z = \core_{\varphi}(M_2) z$. Denote by $\theta^\varphi : \R \curvearrowright \core_{\varphi}(M_2)$ the dual trace scaling action of $\sigma^\varphi : \R \curvearrowright M_{2}$ where we identify $\widehat \R = \R$ (see \cite[Theorem X.2.3]{Ta03}). Then for every $t \in \R$, we have $\core_{\varphi}(B)  \theta_{t}^{\varphi}(z) = \core_{\varphi}(M_2) \theta_{t}^{\varphi}(z)$ and so $\core_{\varphi}(B)  s = \core_{\varphi}(M_2) s$ where $s = \bigvee_{t \in \R} \theta_{t}^{\varphi}(z)$. Observe that $s \in \mathcal Z(\core_{\varphi}(B) ) \cap \mathcal Z(\core_{\varphi}(M_2))$ and $\theta_t^\varphi(s) = s$ for every $t \in \R$. By \cite[Theorem X.2.3 (i)]{Ta03}, there exists a nonzero projection $r \in \mathcal Z(B) \cap \mathcal Z(M_{2})$ so that $\pi_\varphi(r) = s$. We then have $\sigma_t^\varphi(r)=r$ for every $t\in \R$ and $\core_{\varphi}(Br) = \core_{\varphi}(M_{2} r)$. By \cite[Theorem X.2.3 (i)]{Ta03}, we have
	$$ \pi_\varphi(Br) = \core_{\varphi}(Br)^{\theta^\varphi} = \core_\varphi(M_2r)^{\theta^\varphi}=\pi_\varphi(M_2r)$$
and so $B r =  M_{2} r$. 
\end{proof}

\begin{proof}[Proof of Theorem \ref{thmA}]
$(\rm i)$ Since $M_1 \npreceq_{M_1} B$, Theorem \ref{thm-control} implies that $\mathcal Z(M) \subset M_1' \cap M \subset M_1$. Since $B \subset M_2$ is entirely nontrivial, the same argument as in \cite[Theorem 4.3 (1)]{Ue12} shows that $\mathcal Z(M) \subset B$. Thus, we have $\mathcal Z(M) = M' \cap B = \mathcal Z(M_1) \cap \mathcal Z(M_2) \cap \mathcal Z(B)$.

$(\rm ii)$ It is obvious that $\left\{ t \in \R : \exists u \in \mathcal U(B) \text{ such that } \sigma_t^\varphi = \Ad(u) \right \} \subset \rT(M)$. Conversely, let $t \in \rT(M)$. Then there exists $u \in \mathcal U(M)$ such that $\sigma_t^\varphi = \Ad(u)$. Since $M_1 \npreceq_{M_1} B$ and since $u M_1 u^* = \sigma_t^\varphi(M_1) = M_1$, we obtain $u \in \mathcal U(M_1)$ by Theorem \ref{thm-control}. Since $B \subset M_2$ is entirely nontrivial, the same argument as in \cite[Theorem 4.3 (2)]{Ue12} shows that $u \in \mathcal U(B)$. 

$(\rm iii)$ The proof of \cite[Theorem 4.3 (3)]{Ue12} applies {\em mutatis mutandis}.

$(\rm iv)$ By contradiction, assume that there exists a nonzero projection $z \in \mathcal Z(M)$ such that $Mz \lessdot_M M_1$. By item $(\rm i)$, we have $z \in M' \cap B$. Put $Q = Mz \oplus M_1 z^\perp$. Then $M_1 \subset Q \subset M$ is with expectation and $Q \lessdot_M M_1$. Then Theorem \ref{thm-maximal-amenability} implies that $Q \subset M_1$ and so $Mz = M_1 z$. This further implies that $M_2z = Bz$. This however contradicts the assumption that the inclusion $B \subset M_2$ is entirely nontrivial.

$(\rm v)$ Choose a faithful state $\varphi \in M_\ast$ such that $\varphi \circ \rE = \varphi$. Put $\core(B) = \core_\varphi(B)$, $\core(M) = \core_\varphi(M)$ and $\core(M_i) = \core_\varphi(M_i)$ for every $i \in \{1, 2\}$. We canonically have $\core(M) = \core(M_1) \ast_{\core(B)} \core(M_2)$. Using the assumption and Claim \ref{claim-technical}, the exact same reasoning as in item $(\rm i)$ shows that $\mathcal Z(\core(M)) \subset \core(B)$. Thus, we have $\mathcal Z(\core(M)) = \core(M)' \cap \core(B) = \mathcal Z(\core(M_1)) \cap \mathcal Z(\core(M_2)) \cap \mathcal Z(\core(B))$.
\end{proof}

\subsection*{Proof of Theorem \ref{thmB}}

Keep the same notation as before. Let $J$ be any nonempty directed set and ${\omega}$ any cofinal ultrafilter on $J$. Fix a faithful state $\varphi \in M_\ast$ such that $\varphi \circ \rE = \varphi$. For every $i \in \{1, 2\}$, put $Q_i = M_i \vee B^{\omega}$ and $Q = Q_1 \vee Q_2 = M \vee B^\omega$. Observe that $M_{\omega}, B^{\omega}, M_1^{\omega}, M_2^{\omega}, Q_1, Q_2, Q$ are all globally invariant under $\sigma^{\varphi^{\omega}}$ and hence they are all with expectation in $M^{\omega}$. By \cite[Theorem 1.10]{MT13}, we have the following canonical inclusions with trace preserving conditional expectations:
\begin{equation}\label{eq:inclusions}
\core_{\varphi}(M) \subset \core_{\varphi^{\omega}}(M^{\omega}) \subset \core_{\varphi}(M)^{\omega}.
\end{equation}
Put $\mathcal M = \core_\varphi(M)$, $\mathcal B = \core_\varphi(B)$ and $\mathcal M_i = \core_{\varphi}(M_i)$ for every $i \in \{1, 2\}$. We identify $\rL_{\varphi}(\R) = \rL_{\varphi^\omega}(\R)$. Moreover, put $\mathcal Q_i = \mathcal M_i \vee \mathcal B^{\omega}$ for every $i \in \{1, 2\}$ and $\mathcal Q = \mathcal Q_1 \vee \mathcal Q_2 = \mathcal M \vee \mathcal B^\omega$.

\begin{proof}[Proof of Theorem \ref{thmB}]
Fix a faithful normal conditional expectation $\rE_{P} : M \to P$. Choose a faithful state $\psi \in M_{\ast}$ such that $\psi \circ \rE_{P} = \psi$. Write $\core(P) = \Pi_{\varphi, \psi}(\core_{\psi}(P)) \subset \mathcal M$. We have the following trichotomy.

{\bf First case:} Assume that there exist $i \in \{1, 2\}$ and a nonzero finite trace projection $p \in \core(P)$ such that $p\core(P)p \preceq_{\mathcal M} \mathcal M_{i}$. Then we clearly have $\core(P) \preceq_{\mathcal M} \mathcal M_{i}$. This is precisely item $(\rm ii)$ in the statement of Theorem \ref{thmB}.

{\bf Second case:} Assume that there exist a nonzero finite trace projection $p \in \core(P)$ and a nonzero projection $z \in (p\core(P)p)' \cap p\mathcal Mp$ such that $p\core(P)p z \lessdot_{\mathcal M} \mathcal B$. By \cite[Lemma 2.1]{Po81}, we have $(p\core(P)p)' \cap p\mathcal Mp = (\core(P)' \cap \mathcal M)p$. We may choose a projection $q \in \core(P)' \cap \mathcal M$ such that $z = qp$. Then we have $z \in \core(P)q$ and $z \, \core(P)q \, z \lessdot_{\mathcal M} \mathcal B$. Denote by $z_{\core(P)q}(z)$ the central support of the projection $z$ in $\core(P)q$. Observe that $\mathcal Z(\core(P)q) = \mathcal Z(\core(P)) q$ and so $z_{\core(P)q}(z) \in \mathcal Z(\core(P)) q \subset \core(P)' \cap \mathcal M$. Then Lemma \ref{lem-elementary-properties1} $(\rm iv)$ implies that $\core(P)q \, z_{\core(P)q}(z) \lessdot_{\mathcal M} \mathcal B$. Up to replacing $q \in \core(P)' \cap \mathcal M$ by $q z_{\core(P)q}(z) \in \core(P)' \cap \mathcal M$, we may assume that $\core(P)q  \lessdot_{\mathcal M} \mathcal B$. Finally, another application of  Lemma \ref{lem-elementary-properties1} $(\rm iv)$ yields item $(\rm iii)$ in the statement of Theorem \ref{thmB}.

{\bf Third case:} Assume that the {\bf first case} and the {\bf second case} do not occur. For every nonzero finite trace projection $p \in \core(P)$, applying Theorem \ref{thm-appendix} to $\mathcal P = p\core(P)p$, we obtain $(p \core(P) p)' \cap p \mathcal M^{\omega}p = \mathcal P' \cap p \mathcal M^{\omega}p \subset p \mathcal Q p$. By \cite[Lemma 2.1]{Po81}, we have $(p\core(P)p)' \cap p\mathcal M^{\omega}p = (\core(P)' \cap \mathcal M^{\omega})p$ and so $(\core(P)' \cap \mathcal M^{\omega})p \subset p \mathcal Q p$. Since this holds true for every nonzero finite trace projection $p \in \core(P)$, we infer that $\core(P)' \cap \mathcal M^{\omega} \subset \mathcal Q$ and so $\core(P)_{\omega} = \core(P)' \cap \core(P)^{\omega}\subset \core(P)' \cap \mathcal M^{\omega} \subset \mathcal Q$.

Using \cite[Theorem 1.10]{MT13}, for every $(x_{n})^{\omega} \in P_{\omega}$, we have
\begin{equation*}
\pi_{\psi^{\omega}}((x_{n})^{\omega}) = (\pi_\psi(x_{n}))^{\omega} \in \core_{\psi}(P)' \cap \core_{\psi}(P)^\omega = \core_{\psi}(P)_\omega
\end{equation*}
and
\begin{align*}
\pi_{\varphi^{\omega}}((x_{n})^{\omega}) &= \Pi_{\varphi^{\omega}, \psi^{\omega}}(\pi_{\psi^{\omega}}((x_{n})^{\omega})) \\ \nonumber
&= \left( \Pi_{\varphi, \psi} \right)^{\omega} \left((\pi_\psi(x_{n}))^{\omega} \right) \\
& \in  \left( \Pi_{\varphi, \psi} \right)^{\omega} \left( \core_{\psi}(P)_\omega \right) \\
&= \core(P)_\omega.
\end{align*}
Therefore, we have $\pi_{\varphi^{\omega}}(P_{\omega}) \subset \core(P)_\omega$ and so $\pi_{\varphi^{\omega}}(P_{\omega}) \subset \core(P)_\omega \cap \core_{\varphi^{\omega}}(M^{\omega})$.
Denote by $\mathcal E : \mathcal M^\omega \to \core_{\varphi^\omega}(M^\omega)$ the canonical trace preserving conditional expectation arising from \eqref{eq:inclusions}. Then we canonically have $\mathcal E(\mathcal Q) = \core_{\varphi^\omega}(Q)$ and we infer that
\begin{align*}
\pi_{\varphi^\omega}(P_\omega) &= \mathcal E(\pi_{\varphi^\omega}(P_\omega)) \quad (\text{since } \pi_{\varphi^\omega}(P_\omega) \subset \core_{\varphi^\omega}(M^\omega)) \\
& \subset \mathcal E (\core(P)_{\omega}) \quad \quad (\text{since } \pi_{\varphi^\omega}(P_\omega) \subset \core(P)_\omega) \\
& \subset \mathcal E(\mathcal Q) \quad  \quad \quad \; \; \, (\text{since }  \core(P)_\omega \subset \mathcal Q) \\
&= \core_{\varphi^\omega}(Q).
\end{align*}
Denote by $\theta^{\varphi^{\omega}} : \R \curvearrowright \core_{\varphi}(M^{\omega})$ the dual trace-scaling action of $\sigma^{\varphi^{\omega}} : \R \curvearrowright M^{\omega}$ where we identify $\widehat \R = \R$ (see \cite[Theorem X.2.3]{Ta03}). By \cite[Theorem X.2.3 (i)]{Ta03}, we have
	$$ \pi_{\varphi^{\omega}}(P_{\omega}) = \pi_{\varphi^\omega}(P_\omega)^{\theta^{\varphi^{\omega}}}\subset \core_{\varphi^\omega}(Q)^{\theta^{\varphi^\omega}} = \pi_{\varphi^\omega}(Q)$$
and so $P_{\omega} \subset Q$.
\end{proof}

Observe that the proof of Theorem \ref{thmB} shows that Assertion $(\rm ii)$ can be strengthened as follows: there exist $i \in \{1, 2\}$ and a nonzero finite trace projection $p \in \core(P)$ such that $p \core(P) p \preceq_{\mathcal M} \mathcal M_{i}$. We next prove a more precise version of Theorem \ref{thmB} when $P = M$.

\begin{thm}\label{thmB-precise}
Keep the same notation as above. Assume that $M_{1} \npreceq_{M_{1}} B$ and the inclusion $B \subset M_{2}$ is entirely nontrivial. Then at least one of the following assertions holds true:

\begin{itemize}
\item [$(\rm i)$] For every nonempty directed set $J$ and every cofinal ultrafilter ${\omega}$ on $J$, we have
$$M_{\omega} \subset M \vee B^{\omega} \quad \quad \text{and} \quad \quad \core(M)_\omega \subset \core(M) \vee \core(B)^\omega.$$

\item [$(\rm ii)$] There exist $i \in \{1, 2\}$ and a nonzero finite trace projection $p \in \mathcal M$ such that $p\mathcal Mp \preceq_{\mathcal M} \mathcal M_{i}$.
\end{itemize}

\end{thm}

\begin{proof}
We show that Assertion $(\rm iii)$ in Theorem \ref{thmB} cannot hold. By contradiction, if Assertion $(\rm iii)$ in Theorem \ref{thmB} does hold, then there exists a nonzero projection $z \in \mathcal Z(\mathcal M)$ such that $\mathcal Mz \lessdot_{\mathcal M} \mathcal B$ and so $\mathcal Mz \lessdot_{\mathcal M} \mathcal M_{1}$. Then there exists a conditional expectation $\Phi : z \langle \mathcal M, \mathcal M_{1} \rangle z \to \mathcal M z$. 

Denote by $\theta^\varphi : \R \curvearrowright \mathcal M$ the dual trace scaling action of $\sigma^\varphi : \R \curvearrowright M$ where we identify $\widehat \R = \R$ (see \cite[Theorem X.2.3]{Ta03}). By construction of the AFP von Neumann algebra $\mathcal M$, we have $\theta_t^\varphi(\mathcal M_{1}) = \mathcal M_{1}$ and $\theta_t^\varphi \circ \rE_{\mathcal M_{1}} \circ \theta_{-t}^\varphi = \rE_{\mathcal M_{1}}$ for every $t \in \R$. For every $t \in \R$, denote by $V_t \in \mathcal U(\rL^2(\mathcal M))$ the canonical unitary implementing $\theta_t^\varphi \in \Aut(\mathcal M)$. Then we have $\Ad(V_t)(\mathcal M) = \mathcal M$ and $\Ad(V_t)(e_{\mathcal M_{1}}) = e_{\mathcal M_{1}}$ and so $\Ad(V_t)(\langle \mathcal M, \mathcal M_{1}\rangle) = \langle \mathcal M, \mathcal M_{1}\rangle$. We still denote by $\theta_t^\varphi = \Ad(V_t)|_{\langle \mathcal M, \mathcal M_{1}\rangle} \in \Aut(\langle \mathcal M, \mathcal M_{1}\rangle)$ the $\ast$-automorphism that canonically extends $\theta_t^\varphi \in \Aut(\mathcal M)$. Then the conditional expectation $\Phi_t = \theta_t^\varphi \circ \Phi \circ \theta_{-t}^\varphi : \theta_t^\varphi(z) \langle \mathcal M, \mathcal M_{1} \rangle \theta_t^\varphi(z) \to \mathcal M \theta_t^\varphi(z)$ witnesses the fact that $\mathcal M \theta_t^\varphi(z) \lessdot_{\mathcal M} \mathcal M_{1}$. 

Denote by $s = \bigvee_{t \in \R} \theta_t^\varphi(z)$. Then we have $\mathcal Ms \lessdot_{\mathcal M} \mathcal M_{1}$ by Lemma \ref{lem-elementary-properties1} $(\rm v)$. Observe that $s \in \mathcal Z(\mathcal M)$ and $\theta_t^\varphi(s) = s$ for every $t \in \R$. By \cite[Theorem X.2.3 (i)]{Ta03}, there exists a nonzero projection $r \in \mathcal Z(M)$ so that $\pi_\varphi(r) = s$. Observe also that $\sigma_{t}^{\varphi}(r) = r$ for every $t \in \R$ and  $r \in \mathcal Z(M) = M' \cap B$ thanks to Theorem \ref{thmA} $(\rm i)$. Then we have $\core_\varphi(Mr) \lessdot_{{\core_\varphi}(M)} \core_\varphi(M_{1})$ and so $\core_\varphi(Mr) \lessdot_{{\core_\varphi}(Mr)} \core_\varphi(M_{1}r)$ since $\langle {\core_\varphi}(Mr), \core_\varphi(M_{1}r)\rangle = \pi_\varphi(r)\langle {\core_\varphi}(M), \core_\varphi(M_{1})\rangle \pi_\varphi(r)$. We canonically have $(Mr, \rE|_{Mr}) = (M_1r, \rE_1|_{M_1r}) \ast_{Br} (M_2r, \rE_2|_{M_2r})$. Since $\core_\varphi(Mr) \lessdot_{{\core_\varphi}(Mr)} \core_\varphi(M_{1}r)$, we have $Mr \lessdot_{Mr} M_{1}r$ by Theorem \ref{thm-characterization} $(\rm i)$ and so $Mr \lessdot_{M} M_{1}$ since $\langle Mr, M_{1}r\rangle = r\langle M,M_{1}\rangle r$. This however contradicts Theorem \ref{thmA} $(\rm iv)$. 
\end{proof}

\begin{cor}\label{cor-thmB}
Keep the same notation as above. Assume that $M_{1} \npreceq_{M_{1}} B$ and the inclusion $B \subset M_{2}$ is entirely nontrivial. Assume moreover that $\core(M_{1}) \npreceq_{\core(M_{1})} \core(B)$, or $M_{2} \npreceq_{M_{2}} B$ or $M_{2} \lessdot_{M_{2}} B$.

Then for every nonempty directed set $J$ and every cofinal ultrafilter ${\omega}$ on $J$, we have
$$M_{\omega} \subset M \vee B^{\omega} \quad \quad \text{and} \quad \quad \core(M)_\omega \subset \core(M) \vee \core(B)^\omega.$$
\end{cor}

\begin{proof}
We show that Assertion $(\rm ii)$ in Theorem \ref{thmB-precise} cannot hold. By contradiction, if Assertion $(\rm ii)$ in Theorem \ref{thmB-precise} does hold, then there exist $i \in \{1, 2\}$ and a nonzero finite trace projection $p \in \mathcal M$ such that $p\mathcal Mp \preceq_{\mathcal M} \mathcal M_i$. 

If $i = 1$, we have $p\mathcal Mp \preceq_{\mathcal M} \mathcal M_1$ and so $\mathcal M \preceq_{\mathcal M} \mathcal M_{1}$. By Proposition \ref{lem-intertwining-relative-amenability}, there exists a nonzero projection $z \in \mathcal Z(\mathcal M)$ such that $\mathcal Mz \lessdot_{\mathcal M} \mathcal M_1$. We can then proceed as in the proof of Theorem \ref{thmB-precise} to obtain a contradiction. 

If $i = 2$, we have $p\mathcal M p\preceq_{\mathcal M} \mathcal M_2$. We proceed in two different ways.

{\bf (a)} Since $p\mathcal M p\preceq_{\mathcal M} \mathcal M_2$, we have $\mathcal M \preceq_{\mathcal M} \mathcal M_2$. By Proposition \ref{lem-intertwining-relative-amenability}, there exists a nonzero projection $z \in \mathcal Z(\mathcal M)$ such that $\mathcal Mz \lessdot_{\mathcal M} \mathcal M_2$. Applying the proof of Theorem \ref{thmB-precise} to $i = 2$, we obtain that there exists a nonzero projection $r \in \mathcal Z(M)$ such that $Mr \lessdot_{M} M_{2}$. If $M_{2} \npreceq_{M} B$, then we obtain a contradiction using Theorem \ref{thmA} $(\rm iv)$. If $M_{2} \lessdot_{M_{2}} B$, then $M_{2} \lessdot_{M} B$. Since $B \subset M_{1}$ is with expectation, we have $B \lessdot_{M} M_{1}$ and so $M_{2} \lessdot_{M} M_{1}$ using Corollary \ref{cor-transitivity}. By Theorem \ref{thmA} $(\rm i)$, we have $r \in \mathcal Z(M) =  M' \cap B$ and  since $Mr \lessdot_{M} M_{2}$, we have $Mr \oplus M_{2}r^{\perp} \lessdot_{M} M_{2}$. Thus, we obtain that $Mr \oplus M_{2}r^{\perp} \lessdot_{M} M_{1}$ using Corollary \ref{cor-transitivity} and so $Mr \lessdot_{M} M_{1}$. We obtain again a contradiction using Theorem \ref{thmA} $(\rm iv)$.

{\bf (b)} We claim that $p\mathcal Mp  \npreceq_{\mathcal M} \mathcal B$. Indeed, otherwise we have $\mathcal M  \preceq_{\mathcal M} \mathcal B$ and by Proposition \ref{lem-intertwining-relative-amenability}, there exists a nonzero projection $z \in \mathcal Z(\mathcal M)$ such that $\mathcal M z\lessdot_{\mathcal M} \mathcal B$ and so $\mathcal Mz \lessdot_{\mathcal M} \mathcal M_1$. We can then proceed as in the proof of Theorem \ref{thmB-precise} to obtain a contradiction. Since $ p\mathcal M p \preceq_{\mathcal M} \mathcal M_2$ and $p\mathcal M p \npreceq_{\mathcal M} \mathcal B$, proceeding as in the proof of \cite[Proposition 2.6]{BHR12}, there exists a nonzero projection $f \in \mathcal M_2$ such that $f\mathcal Mf = f\mathcal M_2f$. Up to multiplying by elements in $\mathcal M_2$, we may further assume that $f \in \mathcal Z(\mathcal M_2)$ and thus we have $f \mathcal M f = \mathcal M_2 f$. The same reasoning as in the proof of \cite[Theorem 5.7]{HV12} shows that there exists a nonzero projection $e \in \mathcal Z(\mathcal B)$ such that $e \mathcal M_1e = \mathcal B e$. This however contradicts the assumption that $\core(M_{1}) \npreceq_{\core(M_{1})} \core(B)$.

In each of the cases $\core(M_{1}) \npreceq_{\core(M_{1})} \core(B)$, or $M_{2} \npreceq_{M_{2}} B$ or $M_{2} \lessdot_{M_{2}} B$, we have obtained a contradiction. Therefore, Assertion $(\rm ii)$ in Theorem \ref{thmB-precise} cannot hold and thus Assertion $(\rm i)$ in Theorem \ref{thmB-precise} gives the desired conclusion.
\end{proof}

\subsection*{Proofs of Theorem \ref{thmC} and Corollary \ref{corD}}

Keep the same notation as before. 

\begin{proof}[Proof of Theorem \ref{thmC}]
Firsly, we prove that $M_\omega = M' \cap \mathcal Z(B)^\omega$. Choose a faithful state $\varphi \in M_\ast$ such that $\varphi \circ \rE = \varphi$. Then we have $\mathcal Z(B) \subset B^\varphi \subset M^\varphi$ and thus we have $\mathcal Z(B)^\omega \subset (M^\varphi)^\omega \subset (M^\omega)^{\varphi^\omega}$. This implies that $M' \cap \mathcal Z(B)^\omega \subset M' \cap (M^\omega)^{\varphi^\omega} = M_\omega$. It remains to prove the reverse inclusion. Since $B$ is of type ${\rm I}$ and $M_{1} \npreceq_{M_{1}} B$, we have $\core(M_1)\npreceq_{\core(M_1)}\core(B)$ by Proposition \ref{prop-intertwining-semifinite}. By Corollary \ref{cor-thmB}, we have $M_\omega \subset Q$. Observe that
$$
(Q, \rE^{\omega}|_Q) = (Q_1, \rE_1^{\omega}|_{Q_1}) \ast_{B^{\omega}} (Q_2, \rE_2^{\omega}|_{Q_2}).
$$

Since $B$ is of type ${\rm I}$ and $M_{1} \npreceq_{M_{1}} B$, we have $M_1 \npreceq_{Q_1} B^\omega$ by Theorem \ref{thm-intertwining-ultraproduct}. Then we have $M_\omega \subset M_1' \cap Q \subset Q_1$ by Theorem \ref{thm-control}.
Observe that $\rE = \rE^\omega |_M$. Working inside the AFP von Neumann algebra $Q$ and since $B \subset M_2$ is entirely nontrivial, the same argument as in the proof of \cite[Theorem 4.3 (1)]{Ue12} shows that $M_\omega \subset M' \cap B^\omega$. Since $B' \cap B^\omega = \mathcal Z(B)^\omega$, we obtain $M_\omega \subset M' \cap B^\omega = M' \cap \mathcal Z(B)^\omega$.

Secondly, since the inclusion $B \subset M_{2}$ is entirely nontrivial, so is $\mathcal B \subset \mathcal M_{2}$ by Claim \ref{claim-technical}. Since $\mathcal B$ is of type ${\rm I}$ and $M_{1} \npreceq_{M_{1}} B$, we have $\mathcal M_{1} \npreceq_{\mathcal M_{1}} \mathcal B$ by Proposition \ref{prop-intertwining-semifinite} and so $\mathcal M_{1} \npreceq_{\mathcal Q_{1}} \mathcal B^{\omega}$ by Theorem \ref{thm-intertwining-ultraproduct}. Reasoning as in the first paragraph and working inside $\mathcal Q = \mathcal Q_{1} \ast_{\mathcal B^{\omega}} \mathcal Q_{2}$ instead of $Q$, we obtain that $\mathcal M_\omega = \mathcal M' \cap \mathcal Z(\mathcal B)^\omega$.

Furthermore assume that $M$ is a full factor with separable predual. Fix a faithful state $\varphi \in M_{\ast}$ such that $\varphi \circ \rE = \varphi$. Let $(t_{n})_{n}$ be any sequence in $\R$. If there exists a sequence of unitaries $v_{n} \in \mathcal U(B)$ such that $\Ad(v_{n}) \circ \sigma_{t_{n}}^{\varphi} \to \id_{M}$ then $t_{n} \to 0$ with respect to $\tau(M)$. Conversely, assume that $t_{n} \to 0$ with respect to $\tau(M)$. Then there exists a sequence of unitaries $u_{n} \in \mathcal U(M)$ such that $\Ad(u_{n}) \circ \sigma_{t_{n}}^{\varphi} \to \id_{M}$. Observe that $\lim_{n \to \infty} \|u_{n} \varphi - \varphi u_{n}\| = 0$. Denote by $(\lambda_\varphi(t))_{t \in \R}$ the canonical unitaries in $\mathcal M$ implementing the modular group $\sigma^\varphi$. By \cite[Lemma XII.6.14]{Ta03}, we have that $\Ad(\pi_{\varphi}(u_n) \lambda_\varphi(t_n)) \to \id_{\mathcal M}$ with respect to the $u$-topology in $\Aut(\mathcal M)$. Let $\omega \in \beta(\N) \setminus \N$ be any nonprincipal ultrafilter. Then we have $(\pi_{\varphi}(u_{n}) \lambda_{\varphi}(t_{n}))_{n} \in \mathfrak M^{\omega}(\mathcal M)$ and $(\pi_{\varphi}(u_{n}) \lambda_{\varphi}(t_{n}))^{\omega} \in \mathcal M' \cap \mathcal M^{\omega}$. Since $\mathcal M_\omega = \mathcal M' \cap \mathcal Z(\mathcal B)^\omega$, we obtain $(\pi_{\varphi}(u_{n}) \lambda_{\varphi}(t_{n}))^{\omega} \in \mathcal B^{\omega}$. Since $\rL_{\varphi}(\R) \subset \mathcal B$ is with expectation and since $( \lambda_{\varphi}(t_{n}))^{\omega} \in \mathcal U(\rL_{\varphi}(\R)^{\omega})$, \cite[Theorem 1.10]{MT13} implies that 
$$\pi_{\varphi^{\omega}}((u_{n})^{\omega}) = (\pi_{\varphi}(u_{n}) \lambda_{\varphi}(t_{n}))^{\omega} \cdot (( \lambda_{\varphi}(t_{n}))^{\omega})^{*} \in \core_{\varphi^{\omega}}(M^{\omega}) \cap \mathcal B^{\omega}.$$

Reasoning as in the third case of the proof of Theorem \ref{thmB}, we obtain that $(u_{n})^{\omega} \in B^{\omega}$ and so $\lim_{n \to \omega} \|u_{n} - \rE(u_{n})\|_{\varphi} = 0$. Since this holds for every $\omega \in \beta(\N) \setminus \N$, we have that $\lim_{n \to \infty} \|u_{n} - \rE(u_{n})\|_{\varphi} = 0$. Using \cite[Lemma 5.1]{Ma16}, we can find a sequence of invertible elements $b_{n} \in \Ball(B)$ such that $\lim_{n \to \infty} \|u_{n} - b_{n}\|_{\varphi} = 0$. Writing $b_{n} = v_{n} |b_{n}|$ for the polar decomposition of $b_{n} \in B$, we have $v_{n} \in \mathcal U(B)$, $\lim_{n \to \infty} \||b_{n}| - 1\|_{\varphi}$ and so $\lim_{n \to \infty}\|u_{n} - v_{n}\|_{\varphi} = 0$. Since $\lim_{n \to \infty} \|u_{n} \varphi - \varphi u_{n}\| = 0$, $\lim_{n \to \infty} \|v_{n} \varphi - \varphi v_{n}\| = 0$ and so $\lim_{n \to \infty} \|v_{n}u_{n}^{\ast} \varphi - \varphi v_{n}u_{n}^{\ast}\| = 0$. This implies that  $\Ad(v_{n} u_{n}^{\ast}) \to \id_{M}$ in $\Aut(M)$. Therefore, we obtain that $\Ad(v_{n}) \circ \sigma_{t_{n}}^{\varphi}  = \Ad(v_{n} u_{n}^{\ast}) \circ (\Ad(u_{n}) \circ \sigma_{t_{n}}^{\varphi}) \to \id_{M}$ in $\Aut(M)$.
\end{proof}

\begin{proof}[Proof of Corollary \ref{corD}]
By translating Properties (P1), (P2), (P3) into von Neumann algebraic terms, we have that $\rL(\mathcal S)$ is of type ${\rm I}$; $\rL(\mathcal R_{1})$ has no nonzero type ${\rm I}$ direct summand and so $\rL(\mathcal R_{1}) \npreceq_{\rL(\mathcal R_{1})} \rL(\mathcal S)$; the inclusion $\rL(\mathcal S) \subset \rL(\mathcal R_{2})$ is entirely nontrivial. Applying Theorem \ref{thmC}, we have  
$$\rL(\mathcal R)_\omega = \rL(\mathcal R)' \cap \mathcal Z(\rL(\mathcal S))^\omega \quad \quad \text{and} \quad \quad \rL(\core(\mathcal R))_\omega = \rL(\core(\mathcal R))' \cap \mathcal Z(\rL(\core(\mathcal S)))^\omega.$$

Let $(t_{n})_{n}$ be any sequence in $\R$. Following \cite[Definition 2.6]{HMV17}, we have  
$t_{n} \to 0$ wrt to $\tau(\mathcal R)$ if and only if there exists a sequence of unitaries $v_{n} \in \mathcal U(\rL^{\infty}(X))$ such that $\Ad(v_{n}) \circ \sigma_{t_{n}}^{\varphi} \to \id_{\rL(\mathcal R)}$ in $\Aut(\rL(\mathcal R))$. Applying Theorem \ref{thmC}, we obtain $\tau(\mathcal R) = \tau(\rL(\mathcal R))$.
\end{proof}

\subsection*{Proof of Theorem \ref{thmE} and Corollary \ref{corF}}
Keep the same notation as before.

\begin{proof}[Proof of Theorem \ref{thmE}]
For every $i \in \{1, 2\}$, put $M_{i} = B \rtimes \Gamma_{i}$. We use the same notation as in the previous subsection. Since $\Gamma_{1} \curvearrowright B$ is properly outer and $\Gamma_{1} \curvearrowright \mathcal Z(B)$ is recurrent, we have $M_{1} \npreceq_{M_{1}} B$ by Theorem \ref{thm-crossed-products}. Since $|\Gamma_{2}| \geq 2$, the inclusion $B \subset M_{2}$ is entirely nontrivial.  

$(\rm i)$ If $\Gamma_{2}$ is finite, then the map $\Phi : \langle M_{2}, B\rangle \to M_{2} : x \mapsto \frac{1}{|\Gamma_{2}|}\widehat \rE_{B}(x)$ is a faithful normal conditional expectation and in particular $M_{2} \lessdot_{M_{2}} B$. If $\Gamma_{2} \curvearrowright \mathcal Z(B)$ is recurrent and since $\Gamma_{2} \curvearrowright B$ is properly outer, we have that $M_{2} \npreceq_{M_{2}} B$ by Theorem \ref{thm-crossed-products}. 

In each of the cases when $\Gamma_2$ is finite or $\Gamma_{2} \curvearrowright \mathcal Z(B)$ is recurrent, we can apply Corollary \ref{cor-thmB} (since $M_{2} \lessdot_{M_{2}} B$ or $M_{2} \npreceq_{M_{2}} B$) to obtain $M_\omega \subset Q$. Observe that
$$
(Q, \rE^{\omega}|_Q) = (Q_1, \rE_1^{\omega}|_{Q_1}) \ast_{B^{\omega}} (Q_2, \rE_2^{\omega}|_{Q_2}).
$$
Since the inclusion $B\subset M_{1}$ is regular and $M_{1} \npreceq_{M_{1}} B$, we have $M_1 \npreceq_{Q_1} B^\omega$ by Theorem \ref{thm-intertwining-ultraproduct}. Then we have $M_\omega \subset M_1' \cap Q \subset Q_1$ by Theorem \ref{thm-control}.
Observe that $\rE = \rE^\omega |_M$. Working inside the AFP von Neumann algebra $Q$ and since $B \subset M_2$ is entirely nontrivial, the same argument as in the proof of \cite[Theorem 4.3 (1)]{Ue12} shows that $M_\omega \subset M' \cap B^\omega$. 

$(\rm ii)$ Since $\Gamma_{1} \curvearrowright \mathcal B$ is properly outer and $\Gamma_{1} \curvearrowright \mathcal Z(\mathcal B)$ is recurrent, we have $\mathcal M_{1} \npreceq_{\mathcal M_{1}} \mathcal B$ by Theorem \ref{thm-crossed-products}. By Corollary \ref{cor-thmB}, we have $M_\omega \subset Q$ and $\mathcal M_{\omega} \subset \mathcal Q$. Reasoning as in the proof of item $(\rm i)$, we obtain $M_\omega \subset M' \cap B^\omega$. 

Since the inclusion $B \subset M_{2}$ is entirely nontrivial, so is $\mathcal B \subset \mathcal M_{2}$ by Claim \ref{claim-technical}.  Since the inclusion $\mathcal B \subset \mathcal M_{1}$ is regular and $\mathcal M_{1} \npreceq_{\mathcal M_{1}} \mathcal B$, we have $\mathcal M_1 \npreceq_{\mathcal Q_1} \mathcal B^\omega$ by Theorem \ref{thm-intertwining-ultraproduct}. Reasoning as in the proof of item $(\rm i)$ and working inside $\mathcal Q = \mathcal Q_{1} \ast_{\mathcal B^{\omega}} \mathcal Q_{2}$ instead of $Q$, we obtain $\mathcal M_\omega \subset \mathcal M' \cap \mathcal B^\omega$. 

Furthermore, if $M' \cap B^\omega = \C 1$ for some $\omega \in \beta(\N) \setminus \N$, that is, $M$ is full, then for every faithful state $\varphi \in M_{\ast}$ such that $\varphi \circ \rE = \varphi$ and every sequence $(t_{n})_{n}$ in $\R$, the exact same reasoning as in the proof of Theorem \ref{thmC} shows that the following conditions are equivalent:

\begin{itemize}

\item [$\bullet$] $t_{n} \to 0$ with respect to $\tau(M)$.

\item [$\bullet$] There exists a sequence $v_{n} \in \mathcal U(B)$ such that $\Ad(v_{n}) \circ \sigma_{t_{n}}^{\varphi} \to \id_{M}$ in $\Aut(M)$.

\end{itemize}
This finishes the proof.
\end{proof}

\begin{proof}[Proof of Theorem \ref{corF}]
For every $i \in \{1, 2\}$, put $M_{i} = B \rtimes \Gamma_{i}$. We use the same notation as in the previous subsection. Since $|\Gamma_{1}| = +\infty$ and $\mathcal Z(B) = \C 1$, the action $\Gamma_{1} \curvearrowright \mathcal Z(B)$ is recurrent. By Theorem \ref{thmE} $(\rm i)$, for every $\omega \in \beta(\N) \setminus \N$, we have $M_{\omega} \subset M' \cap B^{\omega}$. Since $B$ is full, we have $B' \cap B^{\omega} = \C 1$ by \cite[Theorem 5.2]{AH12} and so $M_{\omega} = \C 1$. Therefore, $M$ is full.

Assume moreover that $B$ is of type ${\rm III_{1}}$ and $\Gamma_{1} \curvearrowright \mathcal B$ is outer. Then $\mathcal B' \cap \mathcal M_{1} = \C 1$. Since $|\Gamma_{1}| = +\infty$ and $\mathcal Z(\mathcal B) = \C 1$, the action $\Gamma_{1} \curvearrowright \mathcal Z(\mathcal B)$ is recurrent. Since $\Gamma_{1} \curvearrowright \mathcal B$ is outer and $\Gamma_{1} \curvearrowright \mathcal Z(\mathcal B)$ is recurrent, we have $\mathcal M_{1} \npreceq_{\mathcal M_{1}} \mathcal B$ by Theorem \ref{thm-crossed-products}. Then Theorem \ref{thm-control} implies that $\mathcal Z(\mathcal M) \subset \mathcal M_{1}' \cap \mathcal M \subset \mathcal M_{1}' \cap \mathcal M_{1} = \C1$. Thus, $M$ is of type ${\rm III_{1}}$. By Theorem \ref{thmE} $(\rm ii)$, we obtain $\tau(M) = \tau(B)$.
\end{proof}

\appendix

\section{Central sequences in semifinite AFP von Neumann algebras}

For every $i \in \{1, 2\}$, let $\mathcal B \subset \mathcal  M_i$ be any inclusion of $\sigma$-finite semifinite von Neumann algebras with trace preserving conditional expectation $\rE_i : \mathcal M_i \to \mathcal B$. Denote by $(\mathcal M, \rE) = (\mathcal M_1, \rE_1) \ast_{\mathcal B} (\mathcal M_2, \rE_2)$ or simply by $\mathcal M = \mathcal M_1 \ast_{\mathcal B} \mathcal M_2$ the corresponding semifinite AFP von Neumann algebra. Let $J$ be any nonempty directed set and ${\omega}$ any cofinal ultrafilter on $J$. For every $i \in \{1, 2\}$, put $\mathcal Q_i = \mathcal M_i \vee \mathcal  B^{\omega}$ and $\mathcal Q = \mathcal Q_1 \vee \mathcal Q_2 = \mathcal M \vee \mathcal B^\omega$.

The main technical result of this appendix is the following semifinite analogue of \cite[Theorem 6.3]{Io12}. Instead of adapting the proof of \cite[Theorem 6.3]{Io12}, we proceed  as in the proof of \cite[Theorem 4.1]{HU15a}.

\begin{thm}[{\cite[Theorem 6.3]{Io12}}]\label{thm-appendix}
Keep the same notation as above. Let $p \in \mathcal M$ be any nonzero finite trace projection and $\mathcal P \subset p\mathcal M p$ any von Neumann subalgebra. 

Then at least one of the following assertions hold true:
\begin{itemize}
\item [$(\rm i)$] We have $\mathcal P' \cap p\mathcal M^{\omega} p \subset p\mathcal Qp$.
\item [$(\rm ii)$] There exists $i \in \{1, 2\}$ such that $\mathcal P \preceq _{\mathcal M} \mathcal M_i$.
\item [$(\rm iii)$] There exists a nonzero projection $z \in \mathcal Z(\mathcal P' \cap p\mathcal Mp)$ such that $\mathcal P z \lessdot_{\mathcal M} \mathcal B$.
\end{itemize}
\end{thm}

\begin{proof}
Assuming that assertions $(\rm ii)$ and $(\rm iii)$ do not hold, we prove that assertion $(\rm i)$ holds. We divide the proof of Theorem \ref{thm-appendix} into a series of claims.

Following \cite{IPP05}, put $\widetilde{\mathcal M} = \mathcal M \ast_{\mathcal B}(\mathcal B \ovt \rL(\F_2))$ which is a semifinite AFP von Neumann algebra. Denote by $(\theta_t)_{t \in \R}$ the trace preserving malleable deformation associated with the semifinite AFP von Neumann algebra $\mathcal M = \mathcal M_1 \ast_{\mathcal B} \mathcal M_2$ (see \cite{BHR12} for further details). Put $\delta_t (x)= \theta_t(x) - \rE_{\mathcal M}(\theta_t(x))$ for every $x \in \mathcal M$. Also, put $\Theta_t (X)= \theta_t^{\omega}(X)$ for every $X \in \widetilde{\mathcal M}^{\omega}$. We simply use the notation $\|\cdot\|_2$ for the $\rL^2$-norm associated with the canonical faithful normal semifinite trace on any of the semifinite von Neumann algebras considered in this appendix. 

Although the map $\R \to \Aut(\widetilde{\mathcal M}^{\omega}) : t \mapsto \Theta_t$ need not be continuous, we nevertheless prove the following uniform convergence result.

\begin{claim}\label{claim1}
We have
\begin{equation}\label{eq-uniform}
\lim_{t \to 0} \sup_{X \in \Ball(\mathcal P' \cap p \mathcal M^{\omega} p)} \|X - \Theta_t(X) \|_2 = 0.
\end{equation}
\end{claim}

\begin{proof}[Proof of Claim \ref{claim1}]
The proof is a straightforward generalization of the one of \cite[Theorem 4.1, Claim]{HU15a} using moreover \cite[Theorem A.1 and Lemma A.2]{HU15b}. We nevertheless give the details for the reader's convenience.  

Assume by contradiction that \eqref{eq-uniform} does not hold. Thus there exist $c > 0$, a sequence $(t_k)_k$ of positive reals such that $\lim_k t_k = 0$ and a sequence $(X_k)_k$ in $\Ball(\mathcal P' \cap p \mathcal M^{\omega} p)$ such that $\|X_k - \Theta_{2t_k}(X_k)\|_2 \geq 2c$ for all $k \in \N$. Write $X_k = (x_{j}^{(k)})^{\omega}$ with $x_{j}^{(k)} \in \Ball(p \mathcal M p)$ satisfying $\lim_{j \to {\omega}} \|y x_j^{(k)} - x_{j}^{(k)} y\|_2 = 0$ and $2c \leq \|X_k - \Theta_{2t_k}(X_k)\|_2 = \lim_{j \to {\omega}} \|x_j^{(k)} - \theta_{2t_k}(x_j^{(k)})\|_2$ for all $k \in \N$ and all $y \in \mathcal P$.

Denote by $I$ the directed set of all pairs $(\varepsilon, \mathcal F)$ with $\varepsilon  > 0$ and $\mathcal F \subset \Ball(\mathcal P)$ finite subset with order relation $\leq$ defined by
$$(\varepsilon_1, \mathcal F_1) \leq (\varepsilon_2, \mathcal F_2) \quad \text{if and only if} \quad \varepsilon_2 \leq \varepsilon_1, \mathcal F_1 \subset \mathcal F_2.$$
Let $i = (\varepsilon, \mathcal F) \in I$ and put $\delta = \min(\frac{\varepsilon}{6}, \frac{c}{8})$. Choose $k \in \N$ large enough so that $\|p - \theta_{t_k}(p)\|_2 \leq \delta$ and $\|a - \theta_{t_k}(a)\|_2 \leq \varepsilon/6$ for all $a \in \mathcal F$. Then choose $j \in J$ large enough so that $\|x_j^{(k)} - \theta_{2t_k}(x_j^{(k)})\|_2 \geq c$ and $\|a x_j^{(k)} - x_j^{(k)} a\|_2 \leq \varepsilon /3$ for all $a \in \mathcal F$.

Put $\xi_i = \theta_{t_k}(x_j^{(k)}) - \rE_{ \mathcal M }(\theta_{t_k}(x_j^{(k)})) \in \rL^2( \widetilde{\mathcal M} ) \ominus \rL^2( \mathcal M )$ and $\eta_i = p \xi_i p \in \rL^2(p \widetilde{\mathcal M} p) \ominus \rL^2(p \mathcal M p)$. By the transversality property of the malleable deformation $(\theta_t)_{t \in \R}$ (see \cite[Lemma 2.1]{Po06}), we have 
$$
    \|\xi_i\|_2 \geq \frac{1}{2} \|x_j^{(k)} - \theta_{2t_k} (x_j^{(k)})\|_2 \geq \frac{c}{2}.
$$
Observe that $\|p\theta_{t_k}(x_j^{(k)})p - \theta_{t_k}(x_j^{(k)})\|_2 \leq 2 \|p - \theta_{t_k}(p)\|_2 \leq 2 \delta$. Since $p \in \mathcal M$, by Pythagoras theorem, we moreover have  
$$\|p\theta_{t_k}(x_j^{(k)})p - \theta_{t_k}(x_j^{(k)})\|_2^2 = \|\rE_{ \mathcal M }(p\theta_{t_k}(x_j^{(k)})p - \theta_{t_k}(x_j^{(k)}))\|_2^2 + \|\eta_i - \xi_i\|_2^2$$ 
and hence $\|\eta_i - \xi_i\|_2 \leq 2\delta$. This implies that 
$$\|\eta_i\|_2 \geq \|\xi_i\|_2 - \|\eta_i - \xi_i\|_2 \geq  \frac{c}{2} - 2 \delta \geq \frac{c}{4}.$$
For all $x \in p \mathcal M p$, we have  
$$\|x \eta_i\|_2 = \|(1 - \rE_{ \mathcal M}) (x \theta_{t_k} (x_j^{(k)})p) \|_2 \leq \|x \theta_{t_k}(x_j^{(k)})p\|_2 \leq \|x\|_2.$$
By Popa's spectral gap argument \cite{Po06}, for all $a \in \mathcal F \subset \Ball(\mathcal P) \subset \Ball(p \mathcal M p)$, we have  
  \begin{align*}
    \| a \eta_i - \eta_i a \|_2
    & =
    \|(1 - \rE_{\mathcal M}) (a \theta_{t_k}( x_j^{(k)} )p - p\theta_{t_k}( x_j^{(k)}) a)\|_2 \\
    & \leq \|a \theta_{t_k}( x_j^{(k)})p - p\theta_{t_k}( x_j^{(k)}) a\|_2 \\
    & \leq 2 \|a - \theta_{t_k}(a)\|_2 + 2 \|p - \theta_{t_k}(p)\|_2 + \|a x_j^{(k)} - x_j^{(k)} a\|_2 \\
    & \leq \frac{\varepsilon}{3} + \frac{\varepsilon}{3} + \frac{\varepsilon}{3}  = \varepsilon.
  \end{align*}
Thus $\eta_i \in \rL^2(p \widetilde {\mathcal M} p) \ominus \rL^2(p \mathcal M p)$ is a net of vectors satisfying $\limsup_i \|x \eta_i\|_2 \leq \|x\|_2$ for all $x \in p \mathcal M p$, $\liminf_i \|\eta_i\|_2 \geq~\frac{c}{4}$ and $\lim_i \|a \eta_i - \eta_i a\|_2 = 0$ for all $a \in \mathcal P$. 

By construction of the AFP von Neumann algebra $\widetilde{\mathcal M} = \mathcal M \ast_{\mathcal B}(\mathcal B \ovt \rL(\F_2))$, there exists a $\mathcal B$-$\mathcal L$-bimodule such that we have $\rL^2(\widetilde {\mathcal M}) \ominus \rL^2(\mathcal M) \cong \rL^2(\mathcal M) \otimes_{\mathcal B} \mathcal L$ as $\mathcal M$-$\mathcal M$-bimodules (see e.g.\ \cite[Section 2]{Ue98}). The existence of the net $(\eta_i)_{i \in I}$ in combination with \cite[Lemma A.2]{HU15b} shows that there exists a nonzero projection $z \in \mathcal Z(\mathcal P' \cap p \mathcal Mp)$ such that $\mathcal P z \lessdot_{\mathcal M} \mathcal B$. This however contradicts our assumption that assertion $(\rm iii)$ does not hold and finishes the proof of the claim.
\end{proof}

By contradiction, assume that $\mathcal P' \cap p \mathcal M^{\omega} p \not\subset p\mathcal Q p$. Let $X \in \mathcal P' \cap p \mathcal M^{\omega} p \setminus p\mathcal Q p$. Put $Y =  X - \rE_{\mathcal Q}(X)$. Observe that $Y \in \mathcal P' \cap p\mathcal M^{\omega} p$, $Y \neq 0$ and $\rE_{\mathcal Q}(Y) = 0$. Denote by $\rE_\omega : \mathcal M^\omega \to \mathcal M$ the canonical faithful normal conditional expectation. Put $y = \rE_{\omega}(Y^*Y) \in (\mathcal P' \cap p \mathcal M p)^+$. Define the nonzero spectral projection $z = \mathbf 1_{[\frac{1}{2}\|y\|_\infty, \|y\|_\infty]}(y) \in \mathcal P' \cap p \mathcal M p$ and put $c = (y z)^{-1/2} \in (\mathcal P' \cap p \mathcal M p)^+$. Then we have
$$\rE_{\omega}((Yc)^*(Yc)) = c \, \rE_{\omega}(Y^* Y) \, c = c \, y \, c = z.$$
Up to replacing $Y$ by $Yc$ which still lies in $\mathcal P' \cap p\mathcal M^{\omega} p$, we may assume that $Y \in \mathcal P' \cap p\mathcal M^{\omega} p$, $Y \neq 0$, $\rE_{\mathcal Q}(Y) = 0$ and $\rE_{\omega}(Y^*Y) = z$. By Claim \ref{claim1}, we may choose $t \in (0, 1/2)$ small enough such that $\max \{ \|Y - \Theta_t(Y)\|_2, \|z - \theta_t(z)\|_2 \} < \frac{1}{64}$.

\begin{claim}\label{claim2}
There exists a uniformly bounded net $(y_i)_{i \in I}$ in $p \mathcal M p$ that satisfies the following properties:
\begin{itemize}
\item [(P1)] $y_i^*y_i \to z$ $\sigma$-weakly as $i \to \infty$.
\item [(P2)] $\lim_{i \to \infty} \|\rE_{\mathcal B}(b^* y_i a)\|_2 = 0$ for all $a, b \in p \mathcal M$.
\item [(P3)] $\lim_{i \to \infty} \|x y_i - y_i x\|_2 = 0$ for all $x \in \mathcal P$.
\item [(P4)] $\sup_{i \in I} \|y_i - \theta_t(y_i)\|_2 \leq \frac{1}{64}$.
\end{itemize}
\end{claim}

\begin{proof}[Proof of Claim \ref{claim2}]
Denote by $I$ the directed set of all tuples $(\delta, \mathcal F, \mathcal G, \Xi)$ where $\delta > 0$ and $\mathcal F \subset p \mathcal M$, $\mathcal G \subset p \mathcal M p$, $\Xi \subset \rL^2(p \mathcal M p)$ are finite subsets, with order relation $\leq$  defined by 
$$(\delta_1, \mathcal F_1, \mathcal G_1, \Xi_1) \leq (\delta_2, \mathcal F_2, \mathcal G_2, \Xi_2) \quad \text{if and only if} \quad \delta_2 \leq \delta_1, \mathcal F_1 \subset \mathcal F_2, \mathcal G_1 \subset \mathcal G_2, \Xi_1 \subset \Xi_2.$$

Write $Y = (y_j)^{\omega}$ where $y_j \in p \mathcal M p$ and $\sup_{j \in J} \|y_j\|_\infty \leq \|Y\|_\infty$. First, we have $z = \rE_{\omega}(Y^*Y) = \sigma\text{-weak} \lim_{j \to {\omega}} y_j^*y_j$. Secondly, for all $a, b \in p \mathcal M$, we have 
$$0 = \rE_{\mathcal B^{\omega}}(b^*\rE_{\mathcal Q}(Y)a) = \rE_{\mathcal B^{\omega}}(\rE_{\mathcal Q}(b^*Ya)) = \rE_{\mathcal B^{\omega}}(b^* Y a) = (\rE_{\mathcal B}(b^* y_j a))^{\omega}.$$
Thirdly, for all $x \in \mathcal P$, we have $0 = \|x Y - Y x\|_2 = \lim_{j \to {\omega}} \|x y_j - y_j x\|_2$.
Finally, we have $\frac{1}{64} > \|Y - \Theta_t(Y)\|_2 = \lim_{j \to {\omega}} \|y_j - \theta_t(y_j)\|_2$.

Fix $i = (\delta, \mathcal F, \mathcal G, \Xi) \in I$. Using the above observations, we may choose $j \in J$ large enough so that the element $y_j \in p \mathcal M p$ that we denote now by $y_i$ satisfies $\|y_i\|_\infty \leq \|Y\|_\infty$; $|\langle (y_i^* y_i - z)\xi, \eta\rangle| \leq \delta$ for all $\xi, \eta \in \Xi$; $\|\rE_{\mathcal B}(b^* y_i a)\|_2 \leq \delta$ for all $a, b \in \mathcal F$; $\|x y_i - y_i x\|_2 \leq \delta$ for all $x \in \mathcal G$; $\|y_i - \theta_t(y_i)\|_2 \leq 1/64$. Therefore, the uniformly bounded net $(y_i)_{i \in I}$ satisfies all the desired properties.
\end{proof}

\begin{claim}\label{claim3}
We have 
$$\sup_{x \in \Ball(\mathcal P z)} \|x - \theta_{2t}(x)\|_2 \leq \frac12.$$
\end{claim}

\begin{proof}[Proof of Claim \ref{claim3}]
Using (P2) and the proof of \cite[Theorem 2.5, Claim]{BHR12}, we obtain
\begin{equation}\label{eq-convergence}
\forall c, d \in p(\widetilde{\mathcal M} \ominus \mathcal M), \quad \lim_{i \to \infty} \|\rE_{\mathcal M}(d^* y_i c)\|_2 = 0.
\end{equation}
Let $x \in \Ball(\mathcal P z)$ be any element. We have
\begin{align*}
\|\delta_t(x)\|_2^2 &= \langle \delta_t(z x), \delta_t(x) \rangle \\
&\leq | \langle z\delta_t( x), \delta_t(x) \rangle | + \|z - \theta_t(z)\|_2 \\
&\leq \lim_{i \to \infty} | \langle y_i\delta_t( x), y_i\delta_t(x) \rangle | + \frac{1}{64} \quad (\text{by (P1)}) \\
& \leq \limsup_{i \to \infty} | \langle \delta_t(y_i x), y_i\delta_t(x) \rangle | + \limsup_{i \to \infty} \|y_i - \theta_t(y_i)\|_2+ \frac{1}{64} \\
& = \limsup_{i \to \infty} | \langle \delta_t(x y_i), y_i\delta_t(x) \rangle | + \limsup_{i \to \infty} \|y_i - \theta_t(y_i)\|_2+ \frac{1}{64} \quad (\text{by (P3)}) \\
& \leq \limsup_{i \to \infty} | \langle \delta_t(x) y_i, y_i\delta_t(x) \rangle | + 2\limsup_{i \to \infty} \|y_i - \theta_t(y_i)\|_2+ \frac{1}{64} \\
& \leq \limsup_{i \to \infty} \left(\|\rE_{\mathcal M} (\delta_t(x) y_i\delta_t(x)^*)\|_2 \cdot \|y_{i}\|_{2}\right) + 2\limsup_{i \to \infty} \|y_i - \theta_t(y_i)\|_2+ \frac{1}{64} \\
&\leq \frac{1}{16}  \quad (\text{by (P4) and \eqref{eq-convergence}}). 
\end{align*}
Combining the above inequality with the transversality property of the malleable deformation $(\theta_t)_{t \in \R}$ (see \cite[Lemma 2.1]{Po06}), we obtain 
\begin{equation*}
\sup_{x \in \Ball(\mathcal P z)} \|x - \theta_{2t}(x)\|_2 \leq \sup_{x \in \Ball(\mathcal P z)} 2 \|\delta_t(x)\|_2 \leq \frac12. \qedhere
\end{equation*} 
\end{proof}

Combining Claim \ref{claim3} with \cite[Theorem 3.3]{BHR12}, we obtain that there exists $i \in \{1, 2\}$ such that $\mathcal Pz \preceq_{\mathcal M} \mathcal M_i$. This further implies that $\mathcal P \preceq_{\mathcal M} \mathcal M_i$ and contradicts the assumption that assertion $(\rm ii)$ does not hold. This finally shows that $\mathcal P' \cap p \mathcal M^{\omega} p \subset \mathcal Q$ and finishes the proof of the theorem.
\end{proof}

\bibliographystyle{plain}

\begin{thebibliography}{HMV17}

\bibitem[AH12]{AH12} {\sc H. Ando, U. Haagerup}, {\it Ultraproducts of von Neumann algebras.} J. Funct. Anal. {\bf 266} (2014), 6842--6913.

\bibitem[BH16]{BH16} {\sc R. Boutonnet, C. Houdayer}, {\it Amenable absorption in amalgamated free product von Neumann algebras.} {To appear in Kyoto J. Math.} {\tt arXiv:1606.00808}

\bibitem[BHR12]{BHR12} {\sc R. Boutonnet, C. Houdayer, S. Raum}, {\it Amalgamated free product type ${\rm III}$ factors with at most one Cartan subalgebra.} Compos. Math. {\bf 150} (2014), 143--174.

\bibitem[BHV15]{BHV15} {\sc R. Boutonnet, C. Houdayer, S. Vaes}, {\it Strong solidity of free Araki--Woods factors.} To appear in Amer. J. Math. {\tt arXiv:1512.04820}

\bibitem[Co72]{Co72} {\sc A. Connes}, {\it Une classification des facteurs de type ${\rm III}$.} Ann. Sci. \'{E}cole Norm. Sup. {\bf 6} (1973), 133--252.

\bibitem[Co74]{Co74} {\sc A. Connes}, {\it Almost periodic states and factors of type ${\rm III_1}$.} J. Funct. Anal. {\bf 16} (1974), 415--445.

\bibitem[Co75a]{Co75a} {\sc A. Connes}, {\it Outer conjugacy classes of automorphisms of factors.} Ann. Sci. \'{E}cole Norm. Sup. {\bf 8} (1975), 383--419.

\bibitem[Co75b]{Co75b} {\sc A. Connes}, {\it Classification of injective factors. Cases ${\rm II_1}$, ${\rm II_\infty}$, ${\rm III_\lambda}$, $\lambda \neq 1$.} Ann. of Math. {\bf 74} (1976), 73--115.

\bibitem[Co78]{Co78} {\sc A. Connes}, {\it On the spatial theory of von Neumann algebras}. J. Funct. Anal. {\bf 35} (1980), 153--164.

\bibitem[FM75]{FM75} {\sc J. Feldman, C.C. Moore}, {\it Ergodic equivalence relations, cohomology, and von Neumann algebras. ${\rm I}$ and ${\rm II}$.} Trans. Amer. Math. Soc. {\bf 234} (1977), 289--324, 325--359.

\bibitem[Ga99]{Ga99} {\sc D. Gaboriau}, {\it Co\^ut des relations d'\'equivalence et des groupes.} Invent. Math. {\bf 139} (2000), 41--98.

\bibitem[Ha73]{Ha73} {\sc U. Haagerup}, {\it The standard form of von Neumann algebras.} Math. Scand. {\bf 37} (1975), 271--283.

\bibitem[Ha77a]{Ha77a} {\sc U. Haagerup}, {\it Operator valued weights in von Neumann algebras}, I. J. Funct. Anal. {\bf 32} (1979), 175--206.

\bibitem[Ha77b]{Ha77b} {\sc U. Haagerup}, {\it Operator valued weights in von Neumann algebras}, II. J. Funct. Anal. {\bf 33} (1979), 339--361.

\bibitem[HI15a]{HI15a} {\sc C. Houdayer, Y. Isono}, {\it Unique prime factorization and bicentralizer problem for a class of type ${\rm III}$ factors.} Adv. Math. {\bf 305} (2017), 402--455.

\bibitem[HI15b]{HI15b} {\sc C. Houdayer, Y. Isono}, {\it Bi-exact groups, strongly ergodic actions and group measure space type ${\rm III}$ factors with no central sequence.} Comm. Math. Phys. {\bf 348} (2016), 991--1015.

\bibitem[HMV16]{HMV16} {\sc C. Houdayer, A. Marrakchi, P. Verraedt}, {\it Fullness and Connes' $\tau$ invariant of type ${\rm III}$ tensor product factors.} {\tt arXiv:1611.07914}

\bibitem[HMV17]{HMV17} {\sc C. Houdayer, A. Marrakchi, P. Verraedt}, {\it Strongly ergodic equivalence relations: spectral gap and type ${\rm III}$ invariants.} To appear in Ergodic Theory Dynam. Systems. {\tt arXiv:1704.07326}

\bibitem[HSV16]{HSV16} {\sc C. Houdayer, D. Shlyakhtenko, S. Vaes}, {\it Classification of a family of non almost periodic free Araki-Woods factors}. To appear in J. Eur. Math. Soc. {\tt arXiv:1605.06057}

\bibitem[HU15a]{HU15a} {\sc C. Houdayer, Y. Ueda}, {\it Asymptotic structure of free product von Neumann algebras.} Math. Proc. Cambridge Philos. Soc. {\bf 161} (2016), 489--516.

\bibitem[HU15b]{HU15b} {\sc C. Houdayer, Y. Ueda}, {\it Rigidity of free product von Neumann algebras.} Compos. Math. {\bf 152} (2016), 2461--2492.

\bibitem[HV12]{HV12} {\sc C. Houdayer, S. Vaes}, {\it Type ${\rm III}$ factors with unique Cartan decomposition.} J. Math. Pures Appl. {\bf 100} (2013), 564--590.

\bibitem[Io12]{Io12} {\sc A. Ioana}, {\it Cartan subalgebras of amalgamated free product ${\rm II_1}$ factors.} With an appendix joint with Stefaan Vaes. Ann. Sci. \'Ecole Norm. Sup. {\bf 48} (2015), 71--130.

\bibitem[IPP05]{IPP05} {\sc A. Ioana, J. Peterson, S. Popa}, {\it Amalgamated free products of $w$-rigid factors and calculation of their symmetry groups.} Acta Math. {\bf 200} (2008), 85--153.

\bibitem[Is17]{Is17} {\sc Y. Isono}, {\it Unique prime factorization for infinite tensor product factors.} {\tt arXiv:1712.00925}

\bibitem[Jo82]{Jo82} {\sc V.F.R. Jones}, {\it Index for subfactors.} Invent. Math. {\bf 72} (1983), 1--25.

\bibitem[JS85]{JS85} {\sc V.F.R. Jones, K. Schmidt}, {\it Asymptotically invariant sequences and approximate finiteness.} Amer. J. Math. {\bf 109} (1987), 91--114.

\bibitem[Ka82]{Ka82} {\sc R.V. Kadison}, {\it Diagonalizing matrices}, Amer. J. Math. {\bf 106} (1984), 1451--1468.

\bibitem[Ko85]{Ko85} {\sc H. Kosaki}, {\it Extension of Jones' theory on index
to arbitrary factors.} J. Funct. Anal. {\bf 66} (1986), 123--140.

\bibitem[Ma16]{Ma16} {\sc A. Marrakchi}, {\it Spectral gap characterization of full type $\mathrm{III}$ factors.} {To appear in J. Reine Angew. Math.} {\tt arXiv:1605.09613}

\bibitem[MT13]{MT13} {\sc T. Masuda, R. Tomatsu}, {\it Classification of actions of discrete Kac algebras on injective factors.} Mem. Amer. Math. Soc. {\bf 245} (2017), no. 1160, ix+118 pp.

\bibitem[McD69]{McD69} {\sc D. McDuff}, {\it Central sequences and the hyperfinite factor.} Proc. London Math. Soc. {\bf 21} (1970), 443--461.

\bibitem[MvN43]{MvN43} {\sc F. Murray, J. von Neumann}, {\it Rings of operators.} ${\rm IV}$.  Ann. of Math. {\bf 44} (1943), 716--808.

\bibitem[Oc85]{Oc85} {\sc A. Ocneanu}, {\em Actions of discrete amenable groups on von Neumann algebras.} Lecture Notes in Mathematics, {\bf 1138}. Springer-Verlag, Berlin, 1985. iv+115 pp.

\bibitem[OP07]{OP07} {\sc N. Ozawa, S. Popa}, {\it On a class of $\rm{II}_1$ factors with at most one Cartan subalgebra.} Ann. of Math. {\bf 172} (2010), 713--749.

\bibitem[Pe06]{Pe06} {\sc J. Peterson}, {\it $\rL^2$-rigidity in von Neumann algebras.} Invent. Math. {\bf 175} (2009), 417--433.

\bibitem[PP84]{PP84} {\sc M. Pimsner, S. Popa}, {\it Entropy and index for subfactors.} Ann. Sci. \'Ecole Norm. Sup. {\bf 19} (1986), 57--106.

\bibitem[Po81]{Po81} {\sc S. Popa}, {\it On a problem of R.V. Kadison on maximal abelian $\ast$-subalgebras in factors.} Invent. Math. {\bf 65} (1981), 269--281.

\bibitem[Po83]{Po83} {\sc S. Popa}, {\it Maximal injective subalgebras in factors associated with free groups.} Adv. Math. {\bf 50} (1983), 27--48.

\bibitem[Po90]{Po90} {\sc S. Popa}, {\it Markov traces on universal Jones algebras and subfactors of finite index.} Invent. Math. {\bf 111} (1993), 375--405.

\bibitem[Po95]{Po95} {\sc S. Popa}, {\it Classification of subfactors and their endomorphisms.} CBMS Regional Conference Series in Mathematics, {\bf 86}. Published for the Conference Board of the Mathematical Sciences, Washington, DC; by the American Mathematical Society, Providence, RI, 1995. x+110 pp.

\bibitem[Po01]{Po01} {\sc S. Popa}, {\it On a class of type ${\rm II_1}$ factors with Betti numbers invariants.} Ann. of Math. {\bf 163} (2006), 809--899.

\bibitem[Po03]{Po03} {\sc S. Popa}, {\it Strong rigidity of ${\rm II_1}$ factors arising from malleable actions of w-rigid groups $\rm I$.} Invent. Math. {\bf 165} (2006), 369--408.

\bibitem[Po06]{Po06} {\sc S. Popa}, {\it On the superrigidity of malleable actions with spectral gap.}  J. Amer. Math. Soc. {\bf 21} (2008), 981--1000. 

\bibitem[Ta03]{Ta03} {\sc M. Takesaki}, {\it Theory of operator algebras. ${\rm II}$.}
Encyclopaedia of Mathematical Sciences, {\bf 125}. Operator Algebras and Non-commutative Geometry, 6. Springer-Verlag, Berlin, 2003. xxii+518 pp.

\bibitem[Ue98]{Ue98} {\sc Y. Ueda}, {\it Amalgamated free products over Cartan subalgebra.} Pacific J. Math. {\bf 191} (1999), 359--392.

\bibitem[Ue00]{Ue00} {\sc Y.~Ueda}, 
{\it Fullness, Connes' $\chi$-groups, and ultra-products of amalgamated free products over Cartan subalgebras.}  Trans. Amer. Math. Soc. {\bf 355} (2003), 349--371.

\bibitem[Ue10]{Ue10} {\sc Y. Ueda}, {\it Factoriality, type classification and fullness for free product von Neumann algebras.} Adv. Math. {\bf 228} (2011), 2647--2671.

\bibitem[Ue11]{Ue11} {\sc Y. Ueda}, {\it On type ${\rm III_1}$ factors arising as free products.} Math. Res. Lett. {\bf 18} (2011), 909--920. 

\bibitem[Ue12]{Ue12} {\sc Y. Ueda}, {\it Some analysis on amalgamated free products of von Neumann algebras in non-tracial setting.} J. London Math. Soc. {\bf 88} (2013), 25--48.

\bibitem[Vo85]{Vo85} {\sc D.-V. Voiculescu}, {\it Symmetries of some reduced free product $\rC^*$-algebras.} Operator algebras and Their Connections with Topology and Ergodic Theory, Lecture Notes in Mathematics {\bf 1132}. Springer-Verlag, (1985), 556--588. 

\bibitem[VDN92]{VDN92} {\sc D.-V. Voiculescu, K.J. Dykema, A. Nica}, {\it Free random variables.} CRM Monograph Series {\bf 1}. American Mathematical Society, Providence, RI, 1992. 

\end{thebibliography}

\end{document}